\documentclass{amsart}

\usepackage{graphicx} 
\usepackage{enumitem}
\usepackage{comment}
\usepackage[firstinits=true,style=alphabetic,maxnames=10]{biblatex}
\usepackage[pdfstartview=FitH]{hyperref}

\newtheorem{theorem} {Theorem}[section] 
\newtheorem{lemma}{Lemma}[section] 
\newtheorem{proposition}{Proposition} [section]
\newtheorem{corollary} {Corollary} [section]
\newtheorem{conjecture} {Conjecture} [section]
\newtheorem*{mconj}{Conjecture 1.1 (Matching Conjecture)}
\newtheorem{claim} {Claim} [section]
\theoremstyle{definition} 
\newtheorem{fact} {Fact} [section]
\newtheorem{remark} {Remark} [section]
\newtheorem{definition} {Definition} [section]
\newtheorem{convention} {Convention} [section]

\newtheorem*{ack}{Acknowledgements}

\newcommand{\cA}{\mathcal{A}}
\newcommand{\cF}{\mathcal{F}}
\newcommand{\cG}{\mathcal{G}}
\newcommand{\cH}{\mathcal{H}}
\newcommand{\cP}{\mathcal{P}}
\newcommand{\cT}{\mathcal{T}}

\newcommand{\ds}{\displaystyle}

\numberwithin{equation}{section}

\bibliography{matching}

\title[Hypergraphs with fixed matching number]{On the maximum number of edges in a hypergraph with given
  matching number}

\author[P. Frankl]{Peter Frankl}
\address{Peter Frankl Office, Tokyo, Japan}
\email[Peter Frankl]{peter.frankl@gmail.com}

\begin{document}

\begin{abstract}
  The aim of the present paper is to prove that the maximum number of
  edges in a 3-uniform hypergraph on $n$ vertices and matching number
  $s$ is
  \[
  \max\Bigl\{ \binom{3s+2}{3},\, \binom{n}{3} - \binom{n-s}{3} \Bigr\}
  \]
  for all $n,s,\, n\geq 3s+2$.
\end{abstract}

\maketitle

\section{Introduction}

Let $[n] = \{1,2,\dots,n\}$ be a finite set and $\cF \subset
\binom{[n]}{k}$ a $k$-uniform hypergraph.  The matching number
$\nu(\cF)$ is the maximum number of pairwise disjoint edges in
$\cF$.  Fixing the matching number, say $s$, there are two very natural
constructions for $k$-graphs with that matching number.
\[
\begin{split}
  \cA_k &= \binom{[ks + k - 1]}{k}, \quad \textrm{and}\\
  \cA_1(n) &= \biggl\{ F \in \binom{[n]}{k} : F \cap [s] \not= \emptyset
  \biggr\}. 
\end{split}
\]

In 1965 Paul Erd\H{o}s made the following.
\begin{mconj}[\cite{E}]\label{conj:fbound} 
  If $\cF \subset \binom{[n]}{k}$ satisfies~$\,\nu(\cF) = s$ then
  \[|\cF| \leq \max\{ |\cA_1(n)|,|\cA_k|\}.
  \]
\end{mconj}

\addtocounter{conjecture}{1}

In the same paper Erd\H{o}s proved the conjecture for $n > n_0(k,s)$.
Let us mention that the conjecture is trivial for $k=1$, and it was
proved for graphs $(k = 2)$ by Erd\H{o}s and Gallai~\cite{EG}.

There were several improvements on the bound $n_0(k,s)$.  Bollob\'as,
Daykin and Erd\H{o}s~\cite{BDE} proved $n_0(k,s) \leq 2k^3s$ and
recently Huang, Loh and Sudakov~\cite{HLS} improved it to $n_0(k,s) \leq
3k^2s$. On the other hand, F\"uredi and the present author proved
$n_0(k,s) \leq cks^2$, although their result still awaits publication.

The aim of the present paper is to prove

\begin{theorem}\label{thm:true}
  The conjecture is true for $k=3$.
\end{theorem}

We should mention that our proof relies partly on ideas from
Frankl-R\"odl-Ruci\'nski~\cite{FRR}, who proved $n_0(3,s) \leq 4s$ and
the recent result of Luczak and Mieczkowska~\cite{LM} who proved the
conjecture for $k=3$, $s > s_0$.

Let us mention that the best general bound, true for all $k,s$ and $n
\geq k(s+1)$ is due to the author (cf.~\cite{F95}) and it says
\begin{equation}\label{eq:best}
  |\cF| \leq s\binom{n-1}{k-1}.
\end{equation}

Note that for $n=k(s+1)$, \eqref{eq:best} reduces to $|\cF| \leq
|\cA_1|$.  This special case, the first non-trivial instance of the
conjecture, was proved implicitly by Kleitman~\cite{Kl}.  The case $s=1$
of \eqref{eq:best} is the classical Erd\H{o}s-Ko-Rado Theorem~\cite{EKR}.

\section{Notation, tools}

For a family $\cH \subset 2^{[n]}$ and an element $i \in [n]$ we define
$\cH(i)$ and $\cH(\bar{i})$ by
\[
\begin{split}
  \cH(i) &= \bigl\{ H - \{i\} : i \in H \in \cH \bigr\},\\
  \cH(\bar{i}) &= \bigl\{ H \in \cH : i \notin H \bigr\}.
\end{split}
\]

For a subset $H = \{h_1,\dots,h_q\}$ we denote it \underline{also} by
$(h_1,\dots,h_q)$ whenever we know for \underline{certain} that $h_1 <
h_2 < \cdots < h_q$.

For subsets $H = (h_1,\dots,h_q)$, $G = (g_1,\dots,g_q)$ we define the
partial order, $\ll$ by
\[
H \ll G \quad \textrm{iff} \quad h_i \leq g_i \quad \textrm{for} \quad 1
\leq i \leq q. 
\]

\begin{definition}
  The family $\cF \subset \binom{[n]}{k}$ is called \underline{stable}
  if $G \ll F \in \cF$ implies $G \in \cF$.
\end{definition}

In Frankl~\cite{F87} (cf. also~\cite{F95}) it was proved that it is
sufficient to prove the Matching conjecture for stable families.
Therefore throughout the paper we assume that $\cF$ is stable and use
stability without restraint.

An easy consequence of stability is the following. Let $\cF \subset
\binom{[n]}{k}$,~$\nu(\cF) = s$ and define $\cF_0 = \bigl\{ F \cap
[ks+k-1] : F \in \cF \bigr\}$.  Note that $\cF_0$ is not $k$-uniform in
general.

\begin{proposition}\label{prop:nu}
  $\nu(\cF_0) = s$.
\end{proposition}

\begin{proof}
  Suppose for contradiction that $G_1,\dots,G_{s+1} \in \cF_0$ are
  pairwise disjoint and $F_1,\dots,F_{s+1} \in \cF$ are such that $F_i
  \cap [ks+k-1] = G_i$, $1 \leq i \leq s+1$.  Suppose further that
  $F_1,\dots,F_{s+1}$ are chosen subject to the above condition to
  minimize
  \begin{equation}\label{eq:inter}
    \sum_{1 \leq i < j \leq s+1}|F_i \cap F_j|
  \end{equation}
  Since $\nu(\cF) < s+1$, the above minimum is positive.  We establish
  the contradiction by showing that one can diminish it.

  Choose some $x \in F_i \cap F_j$.  Since $G_i \cap G_j = \emptyset$,
  $x \geq k(s+1)$.  Consequently, $|G_1|+ \cdots + |G_{s+1}| \leq
  k(s+1)-2 < ks + k - 1$.  Thus we can choose $y \in [ks+k-1]$ with $y
  \notin G_i$ for $1 \leq i \leq s+1$.  Now replace $F_i$ by $F_i' =
  (F_i - \{x\}) \cup \{y\}$.  Then $F_i' \ll F_i$, implying $F_i' \in
  \cF$.

  The intersections $F_j \cap[ks+k-1]$, $j=1,\dots,s+1$, $j \not=
  i$ and $F_i' \cap [ks+k-1]$ are still disjoint but the value of
  \eqref{eq:inter} is smaller.
\end{proof}

From now on we shall assume that $\cF \subset \binom{[n]}{k}$ satisfies
$\nu(\cF) = s$ and it is maximal, i.e.,~it cannot be extended without
increasing $\nu(\cF)$.  Then the following formula is evident from
Proposition \ref{prop:nu}.
\begin{equation}\label{eq:sizef}
  |\cF| = \sum_{H \in \cF_0}\binom{n-ks-k+1}{k-|H|}.
\end{equation}

Formula~\ref{eq:sizef} shows that for a fixed $k$ and $s$, determining
$\max{|\cF|}$ is a finite problem, i.e.,~it is sufficient to compare all
families $\cF_0 \subset 2^{[ks+k-1]}$ with $\ds \max_{H \in \cF_0}{|H|}
\leq k$ and $\nu(\cF_0)=s$.

However, this finiteness is only theoretical.  There are too many
families to check.  Let us consider the following families, first
defined in the author's Ph.D.~dissertation in 1976.
\[
\cA_\ell(n) = \biggl\{ F \in \binom{[n]}{k} : \bigl|F \cap [\ell s + \ell -
1] \bigr| \geq \ell \biggr\}.
\]
Then $\nu(\cA_\ell(n)) = s$ holds for $n \geq ks$.

Unless the next proposition holds, we get a counterexample to
Conjecture~\ref{conj:fbound}.

\begin{proposition}\label{prop:ce}
  For all $1 \leq \ell \leq k$,
  \begin{equation}\label{eq:bd}
    |\cA_\ell(n)| \leq \max\{|\cA_1(n)|,|\cA_k|\}.
  \end{equation}
\end{proposition}
In the present paper we only need the validity of
Proposition~\ref{prop:ce} for the case $k=3$.  In that case it is not
hard to check by direct calculation.

\section{Preliminaries}

For a family $\cF \subset \binom{[n]}{k}$, $\nu(\cF)=s$, $n \geq ks+k-1$
we want to define a specific partition
\begin{equation}
  \label{eq:parts}
  F_0 \cup F_1 \cup \cdots \cup F_s =
  [ks+k-1] \quad \textrm{where} \ F_1,\dots,F_s \in \cF.
\end{equation}
Since $\nu(\cF) = s$, we can choose $F_1,\dots,F_s \in \cF$ with~$F_1
\cup \cdots \cup F_s = [ks]$.  Then $F_0 = [ks+1,ks+k-1]$.  However,
we fix $F_0$ to be the \underline{lexicographically} first
$(k-1)$-element subset of $[ks+k-1]$ for which a partition of
type~\eqref{eq:parts} is possible.  Note that $F_0 \notin \cF_0$.  Once
$F_0 = \{d_1,d_2,\dots,d_{k-1}\}$ is fixed we choose $F_i =
\bigl(a_1(i),\dots,a_k(i)\bigr)$ such that $\ds \sum_{1 \leq i
  \leq s} a_1(i)$ is minimal. Once this minimum value is attained we
minimize $\ds \sum_{1 \leq i \leq s}a_2^{(i)}$ and so on.

\begin{proposition}\label{prop:lex}
  For every $1 \leq \ell < k$ and every $(e_1,\dots,e_\ell)$ which
  precedes $(d_1,\dots,d_\ell)$ lexicographically, $(e_1,\dots,e_\ell)
  \in \cF_0$ holds.  
\end{proposition}

\begin{proof}
  Since $\cF$ is maximal, the contrary would mean that there exist
  pairwise disjoint sets $F_1,\dots,F_s \in \cF$ which are disjoint to
  $(e_1,\dots,e_\ell)$ as well.  However, then $(e_1,\dots,e_\ell)$ can
  be extended to a $(k-1)$-element set $D$, which is still disjoint to
  $F_1,\dots,F_s$ and precedes $F_0$ lexicographically, a contradiction.
\end{proof}

The following statement is rather simple to prove, but it is extremely
useful.

\begin{claim}\label{cl:easy}
  Let $h$, $1 \leq h < k$ be the smallest number, $\ell$, such that
  $a_\ell(i) < d_\ell$ holds, and let $h = k$ if no such $\ell$
  exists.  Then
  \[
  D
  \stackrel{\textrm{def}}{=}(d_1,\dots,d_{h-1},a_h(i),d_h,\dots,d_{k-1})
  \in \cF
  \]
  holds.
\end{claim}

\begin{proof}
  If $h < k$ then $(d_1,\dots,d_{h-1},a_h^{(i)}) \in \cF_0$ from
  Proposition~\ref{prop:lex}. Thus \underline{all} $k$-sets containing
  it are in $\cF$.

  If $h = k$ then $D \ll F_i$ implies the claim.
\end{proof}

The next claim can be easily verified using the definitions.

\begin{claim}\label{cl:def}
  For $\cF = \cA_\ell(n)$,
  \[
  F_0(\cA_\ell(n)) = (1,\dots,\ell-1,\ell s + \ell, \ell s + \ell +
  1,\dots,\ell s + k - 1)
  \]
\end{claim}\qed

Let $R = (r_1,\dots,r_p) \subset [s]$ be a $p$-tuple (we assume $k \geq
p$ here).  Define the set $X(R)$ by $X(R) = F_0 \cup F_{r_1} \cup \cdots
\cup F_{r_p}$.  Note that $|X(R)| = kp+k-1$.  Define the
\underline{restriction} $\cH(R) = \{H \in \cF_0 : H \subset X(R)\}$.

\begin{definition}\label{def:width}
  The \underline{width} $v(H)$ of $H \in \cH(R)$ is defined by
  \[
  v(H) = \bigl|\{\ell : H \cap F_{r_{ _\ell}} \not= \emptyset\}\bigr|.
  \]
\end{definition}
Note that $F_0 \notin \cF_0$ implies $v(H) > 0$.  Next we define the
weight of $\cH$.

\begin{definition}\label{def:weight}
  The \underline{weight} $w(H)$ of $H \in \cH(R)$ is defined by
  \[
  w(H) = \frac{\binom{n-ks-k+1}{k-|H|}}{\binom{s-v(H)}{k-v(H)}}.
  \]
\end{definition}
The weight of a $k$-tuple $R = (r_1,\dots,r_k) \subset [s]$ is defined
by
\begin{equation}
  \label{eq:weight}
  \sum_{H \in \cH(R)} w(H)
\end{equation}

These definitions are justified by:

\begin{lemma}[Counting Lemma]
  For $s \geq k$,
  \[
  |\cF| = \sum_{R \in \binom{[s]}{k}} \sum_{H \in \cH(R)}w(H)
  \]
\end{lemma}

\begin{proof}
  In view of~\eqref{eq:sizef} it is sufficient to note that each $H \in
  \cF_0$ is contained in $\cH(R)$ for exactly $\binom{s-v(H)}{k-v(H)}$
  $k$-tuples $R$.
\end{proof}

It is easy to check that for $\cF = \cA_\ell(n)$
\[
\bigcup_{1 \leq \ell \leq s}(a_1(i),a_2(i),\dots,a_\ell(i)) =
[\ell,\ell s + \ell - 1]
\]
holds.  Consequently, the value of~\eqref{eq:weight} is independent of
the particular choice of $R \subset [s]$.  Let $f(\ell)$ denote this
common value.

\begin{conjecture}\label{conj:bound}
  If $\cF \subset \binom{[n]}{k}$, $\nu(\cF) = s$, $\nu(\cF(\bar{1})) =
  s$, $s \geq k$, then
  \begin{equation}
    \label{eq:conj}
    \sum_{H \in \cH(R)}w(H) \leq \max_{1 \leq \ell \leq k}f(\ell)
  \end{equation}
  holds.
\end{conjecture}

One can show that Conjecture~\ref{conj:bound} would imply Erd\H{o}s'
Conjecture~\ref{conj:fbound} for $s \geq k$.  We prove
Theorem~\ref{thm:true} by establishing Conjecture~\ref{conj:bound} for
$k=3$, and certain values of $n$.  Those values are $n = n_0(s,3)$ and
$n = n_0(s,3)-1$ and will be defined later.

The paper is organized as follows.  In Section~\ref{sec:facts} we prove
some easy results, and consider $\cH(R)$ with $|R|=1$.
Section~\ref{sec:ind} provides the foundation for induction.  In
Section~\ref{sec:struc} we consider $\cH(R)$ with $|R| = 2$, $k=3$.  In
Section~\ref{sec:cases} we prove some general results.

In the later sections we concentrate on the case $k=3$.  In
Section~\ref{sec:s2} we show that Conjecture~\ref{conj:fbound} holds for
$s=2$.  In Sections~\ref{sec:fat},\ref{sec:sfat} and~\ref{sec:notfat} we
establish the validity of~\eqref{eq:conj} in the necessary range
settling Conjecture~\ref{conj:fbound} for~$s \geq 4$.  Section
\ref{sec:last} handles the last remaining case,~$s=3$.

\section{Some easy facts}\label{sec:facts}

The property of $\cH(R)$ that we use most is

\begin{fact}
  $\nu(\cH(R)) = |R|$.
\end{fact}

\begin{proof}
  For $R = (r_1,\dots,r_p)$ the family $\cH(R)$ contains
  $F_{r_1},\dots,F_{r_p}$ showing $|\nu(\cH(R))| \geq |R|$.  On the other hand,
  for $1 \leq i \leq s$, $i \notin R$ the edges $F_i \in \cF$ are
  pairwise disjoint and disjoint to the vertex set of $\cH(R)$ as well
  showing $\nu(\cH(R)) + s - |R| \leq \nu(\cF) = s$, proving
  $\nu(\cH(R)) \leq |R|$.
\end{proof}

Let now $k=3$ and $R = \{i\}$, $F_i = (a_i,b_i,c_i)$.

\begin{fact}\label{fact:notin}
  If $d_1=1$ then $(a_i,c_i) \notin \cF_0$, $(b_i,c_i) \notin \cF_0$.
  Moreover, if $(a_i,b_i) \in \cF_0$ then $(1,c_i) \notin \cF_0$.
\end{fact}

\begin{proof}
  Since $(1,b_i) \ll (a_i,c_i)$, $(a_i,c_i) \in \cF_0$ would imply
  $(1,b_i) \in \cF_0$.  This would contradict $\nu(\cH(\{i\})) = 1$. Now
  $(a_i,c_i) \ll (b_i,c_i)$ implies $(b_i,c_i) \notin \cF_0$.  The last
  statement is a direct consequence of $\nu(\cH(\{i\})) = 1$.
\end{proof}

\begin{fact}\label{fact:f0}
  If $(d_1,x_i) \in \cF_0$ then $(F_i - \{x_i\}) \cup \{d_2\}$ is not in
  $\cF_0$. 
\end{fact}\qed

The following easy fact will prove extremely useful in the sequel.

\begin{fact}
  For every $1 \leq i \leq s$,
  \[
\{1, d_2, b_i \} \in \cH(\{i\}).
  \]
\end{fact}

\begin{proof}
  We apply Claim~\ref{cl:easy}.  If $b_i < d_2$, then $h=2$, and
  $(1,b_i,d_2) \in \cH(\{i\})$ is a direct consequence of
  Claim~\ref{cl:easy}.  If $d_2 < b_i$ then Claim~\ref{cl:easy} yields
  $(1,d_2,c_i) \in \cH(\{i\})$.  The statement follows from $(1,d_2,b_i)
  \ll (1,d_2,c_i)$.
\end{proof}

\begin{fact}
  For any two edges $F_u, F_v$ of the special matching $a_i(u) < a_k(v)$ holds.
\end{fact}

\begin{proof}
  The contrary means
  \[
  a_1(v) < a_2(v) < \cdots < a_k(v) < a_1(u) < \cdots < a_k(u).
  \]
  By stability, $(a_1(v),\dots,a_{k-1}(v),a_1(u))$ and
  $(a_k(v),a_2(u),\dots,a_k(u))$ are in $\cF$.  Using these two sets
  instead of $F_u, F_v$ in the special matching decreases $a_1(1) +
  \cdots + a_s(1)$, a contradiction.
\end{proof}

In later sections we are going to compare the total weight
\[
\sum_{H \in \cH(R)}w(H)
\]
for $R \in \binom{[s]}{3}$ with the corresponding weights for $\cA_3$ and
$\cA_2(n)$, (possibly adding a constant).

Suppose $d_1 = 1$ and set $d = d_2$. For $\cA_3$, the corresponding
hypergraph $\cH^{(3)}(\{i\})$ is the complete 3-graph
$\binom{F_i \cup (1,d)}{3}$.  For $\cA_2(n)$ one has
\[
\cH^{(2)}(\{i\}) = \binom{(1,a_i,b_i)}{2} \cup \biggl\{ H \in \binom{F_i
  \cup (1,d)}{3} : \bigl| H \cap (1,a_i,b_i) \bigr| \geq 2 \biggr\},
\]
it consists of 3 sets of size 2 and 7 of size 3.  We are always fixing
$\cA_3$ or $\cA_2(n)$ as our reference, and consider an edge in $\cH(R)$
that is not in the reference hypergraph a \underline{loss}, and an edge
in the reference hypergraph that is not in $\cH(R)$ a \underline{gain}.
Adding with weights the losses and subtracting the weighted sum of gains
is called the \underline{balance}.

In the case $k=3$, we define $\cG = \{ G \in \cF_0 : |G| = 2\}$.

\begin{convention}\label{conv:comp}
  For $G \in \cG$ with width 1, i.e., $G \in \cH(\{i\})$ for some $i$, we
  always consider $G$ together with its complement $(1,t) \cup F_i - G$.
  Since $\nu(\{i\})=1$, not both can be in $\cH(\{i\})$.
\end{convention}

\begin{corollary}
  The balance (real loss) coming from an extra $G \in \cH(i)$, $|G| = 2$
  is never more than
  \[
  \frac{n-3s-2}{\binom{s-1}{2}} - \frac{1}{\binom{s-1}{2}} =
  \frac{n-3s-3}{\binom{s-1}{2}} 
  \]
\end{corollary}\qed

\section{Why induction would work}\label{sec:ind}

For $n \geq ks+k-1$ let $m(n,k,s)$ denote the maximum possible size of
$|\cF|$ over all $\cF \subset \binom{[n]}{k}$ with $\nu(\cF) = s$.

Note the obvious inequality $\nu(\cF_1 \cup \cF_2) \leq \nu(F_1) +
\nu(\cF_2)$. Let us use it to prove:

\begin{fact}\label{fact:ind}
  $\ds m(n,k,s) \leq m(n-1,k,s-1) + \binom{n-1}{k-1}$.
\end{fact}

\begin{proof}
  Let $\cF_1 \subset \binom{[2,n]}{k}$ satisfy $\nu(\cF_1) = s-1$ and
  $|\cF_1| = m(n-1,k,s-1)$.  Define $\cF_2 = \bigl\{F \in \binom{[n]}{k}
  : 1 \in F \bigr\}$.  Now $|\cF_1 \cup \cF_2| = m(n-1,k,s-1) +
  \binom{n-1}{k-1}$ and $\nu(\cF_1 \cup \cF_2) \leq s-1+1 = s$.
\end{proof}

Fact~\ref{fact:ind} would provide us with a counterexample to
Conjecture~\ref{conj:fbound}, should the following be false.
Fortunately, it is true.

\begin{proposition}\label{prop:fort}
  \begin{equation}
    \label{eq:fort}
    \footnotesize
    \max \Bigl\{ \binom{ks+k-1}{k},\binom{n}{k} - \binom{n-s}{k}
    \Bigr\} \geq
    \max\Bigl\{\binom{ks-1}{k},\binom{n-1}{k}-\binom{n-s}{k}\Bigr\} +
    \binom{n-1}{k-1}.
  \end{equation}
\end{proposition}

\begin{proof}
  If the maximum on the RHS is given by $\binom{n-1}{k} -
  \binom{n-s}{k}$ then~\eqref{eq:fort} follows from
  \[
  \binom{n-1}{k} + \binom{n-1}{k-1} - \binom{n-s}{k} = \binom{n}{k} -
  \binom{n-s}{k}.
  \]
  Assume $\binom{n-1}{k} - \binom{n-s}{k} < \binom{ks-1}{k}$.  We claim
  that $n < (k+1)s$.  Indeed for $n = (k+1)s$ one has
  \[
  \frac{\binom{(k+1)s-1}{k}}{\binom{ks-1}{k}} = \prod_{\ell=1}^k
  \frac{(k+1)s-\ell}{ks-\ell} > \biggl(\frac{k+1}{k}\biggr)^k > 2, \quad
  \textrm{for $k \geq 2$.}
  \]
  Consequently for $n=(k+1)s, \ \binom{n-1}{k} -
  \binom{ks-1}{k} > \binom{ks-1}{k}$ holds.

Using the monotonicity of $\binom{n-1}{k-1}$ it is sufficient to prove
\[
\binom{ks-1}{k} + \binom{(k+1)s-2}{k-1} < \binom{(k+1)s-1}{k}.
\]
However it is evident from $ks-1 < (k+1)s-2$ and
\[
\binom{(k+1)s-2}{k} + \binom{(k+1)s-2}{k-1} = \binom{(k+1)s-1}{k}.
\]
\end{proof}

\begin{corollary}
  If for a given $k$, $\cF$ is a minimal counterexample to
  Conjecture~\ref{conj:fbound}, then $\nu(\cF(\bar{1})) = s$ must hold.
\end{corollary}

\begin{proof}
  Suppose $\nu(\cF(\bar{1})) = s-1$.  By minimality, $\cF(\bar{1}) = \{F
  \in \cF : 1 \notin F\}$ is not a counterexample to
  Conjecture~\ref{conj:fbound}.  Also, for $\cF_2 = \{F \in \cF : 1 \in
  F\}$, $|\cF_2| \leq \binom{n-1}{k-1}$ is evident.  By
  Proposition~\ref{prop:fort}, $\cF$ is not a counterexample.
\end{proof}

We have showed now that in an inductive proof of
Conjecture~\ref{conj:fbound}, one can always assume that
$\nu(\cF(\bar{1}))=s$. Reformulating and elaborating:

\begin{fact}~
  
  \begin{itemize}
  \item[(i)]
    $|F| \geq 2$ for all $F \in \cF_0$

  \item[(ii)]
    For $F_0 = (d_1,\dots,d_{k-1}), \ d_1 = 1$ holds.
  \end{itemize}
\end{fact}

\begin{proof}
  Should (i) fail then by stability $\{1\} \in \cF_0$.  Since
  $\nu(\cF(\bar{1}))=s$, we can find $H_1,\dots,H_s \in \cF(\bar{1})$,
  that are pairwise disjoint.  Now the $s+1$ sets $\{1\}, \, H_i
  \cap[ks+k-1], i=1,\dots,s$ form a matching of size $s+1$ in $\cF_0$,
  contradicting Proposition~\ref{prop:nu}.
\end{proof}

\begin{proposition}
  Suppose that Conjecture~\ref{conj:fbound} holds for $(n-1,k-1,s)$ and
  $(n-1,k,s)$.  Moreover, for $(n-1,k,s)$ the maximum is given by
  $\cA_1(n-1)$.  Then Conjecture~\ref{conj:fbound} holds for $(n,k,s)$
  and the maximum is given by $\cA_1(n)$.
\end{proposition}

\begin{proof}
  Consider the two families $\cF(n)$ and $\cF(\bar{n})$.  By
  Proposition~\ref{prop:nu}, $\nu(\cF(n)) \leq s$ holds.  For
  $\cF(\bar{n})$, $\nu(\cF(\bar{n})) \leq \nu(\cF) \leq s$ is evident.
  By the hypothesis $|\cF(\bar{n})| \leq \binom{n-1}{k} -
  \binom{n-s-1}{k}$.

  On the other hand, we showed above that for $n \geq ks$,
  $|\cA_1(n-1,k-1)| > \binom{(k-1)(s+1)-1}{k}$, thus $|\cF(n)| \leq
  \binom{n-1}{k-1} - \binom{n-s-1}{k-1}$.

  Now $|\cF| = |\cF(n)| + |\cF(\bar{n})|$ yields $|\cF| \leq
  \binom{n}{k} - \binom{n-s}{k}$.
\end{proof}

\begin{definition}
  For $k$ and $s$ fixed let $n_0(s,k)$ be the minimum integer $n$, such
  that $|\cA_k| \leq |\cA_1(n)|$ holds.  Then $n_0(s,k)$ is called the
  \underline{pivotal} number for $k$ and $s$.
\end{definition}

Above we showed $n_0(s,k) < (k+1)s$.

\begin{proposition}
  $\ds n_0(s,k) \leq \biggl( k + \frac{1}{2} \biggr)s + k$
\end{proposition}

\begin{proof}
  First note that setting
  $m = \bigl\lfloor \bigl( k+\frac{1}{2} \bigl)s + k \bigr\rfloor$ we
  have $m \geq \bigl( k + \frac{1}{2} \bigr)s + k - \frac{1}{2}$.  We
  have to show,
  \[
  \binom{m}{k} - \binom{m-s}{k} \geq \binom{k(s+1)-1}{k}.
  \]
  The right hand side is $s\binom{k(s+1)-1}{k-1}$.  The left hand side
  can be estimated using the convexity of $\binom{x}{k-1}$ by Jensen's
  inequality.
  \[
  \binom{m}{k} - \binom{m-s}{k} = \sum_{i=1}^s \binom{m-i}{k-1} > s
  \binom{m-\frac{s}{2}-\frac{1}{2}}{k-1}.
  \]
  Since $\bigl( k+\frac{1}{2} \bigr)s + k - \frac{1}{2} - \frac{s}{2} -
  \frac{1}{2} = k(s+1)-1$, the statement follows.
\end{proof}

Noting that $\cA_3 = \binom{[ks+k-1]}{k}$ does not depend on $n$, we see
that proving $m(n,k,s) \leq \binom{ks+k-1}{s}$ for $n=n_0(s,k)$ implies
the same for all $n < n_0(s,k)$ as well.  Since $m(n,2,s) = \binom{n}{2}
- \binom{n-s}{2}$ is an old theorem of Erd\H{o}s and Gallai~\cite{EG}
for $n \geq 3s$, we infer

\begin{fact}
  In order to prove Conjecture~\ref{conj:fbound} for $k=3$, it is
  sufficient to show it for $n = n_0(s,3)$ and $n=n_0(s,3)-1$.
\end{fact}\qed

\section{The structure of $\cH(i,j)$}\label{sec:struc}
In this section we let $k=3$ and $R=(i,j)$.  Let
\[
\cH_\ell = \bigl\{ H \in \cH(i,j) : |H| = 2, v(H)=\ell \bigr\},~\ell=1,2.
\]
In the previous section we proved $1 \in F_0$.  To simplify notation we
set $d = d_2$, i.e., $F_0=(1,d)$.

\begin{proposition}\label{prop:H2}
  If $|\cH_2| \geq 3$ then one of the following holds.
  \begin{itemize}
  \item[(i)] $\cH_2 = \bigl\{ (a_i,a_j), (a_i,b_j), \{b_i,a_j\},
    \{b_i,b_j\} \bigr\}$,
  \item[(ii)] $\cH_2 = \bigl\{ (a_i,a_j), (a_i,b_j), \{b_i,a_j\}
    \bigr\}$,
  \item[(iii)] $\cH_2 = \bigl\{ (a_i,a_j), (a_i,b_j), (a_i,c_j) \bigr\}$.
  \end{itemize}
\end{proposition}

\begin{proof}
  First of all $(a_i,c_i) \ll (a_j,c_i)$ and Fact~\ref{fact:notin} imply
  $(a_j,c_i) \notin \cF_0$.

  If $(a_i,c_j) \notin \cF_0$, then stability implies that (i) or (ii)
  hold.

  If $(a_i,c_j) \in \cF_0$ then $(a_i,b_j), (a_i,a_j) \in \cF_0$ follow
  by stability.  We claim that $\{b_i,a_j\} \notin \cF_0$.  Indeed,
  otherwise using $(1,b_j) \ll (a_i,c_j)$ we find three pairwise
  disjoint sets $\{b_i,a_j\}, (1,b_j), (a_i,c_j) \in \cH(i,j)$,
  contradicting $\nu(\cH(i,j)) = 2$.  By stability, (iii) holds.
\end{proof}

\begin{fact}\label{fact:51}
  In cases (i) and (ii) neither $\{1,c_i,c_j\}$ nor $(1,c_i)$, nor
  $(1,c_j)$ is in $\cF_0$.  Also neither $(a_i,d,c_j)$ nor $(a_j,d,c_i)$
  is in $\cF_0$.
\end{fact}

\begin{proof}
  Since $(a_i,b_j)$ and $\{b_i,a_j\}$ are in $\cH(i,j)$, $\{1,c_i,c_j\}
  \notin \cF_0$, $(1,c_i) \notin \cH_1$ and $(1,c_j) \notin \cH_1$ are
  direct consequences of $\nu(\cH(i,j))=2$.  $(a_i,d,c_j), (a_j,d,c_i)
  \notin \cF_0$ follow similarly, using $(1,b_j) \in \cH_1$ and $(1,b_i)
  \in \cH_1$.
\end{proof}

\begin{corollary}\label{cor:missing5}
  In cases (i) and (ii) the five sets of width 2, $\{x_i,d,c_j\} : x_i
  \in F_i$, $(a_j,d,c_i), (b_j,d,c_i)$ are all missing from $\cH(i,j)$.
\end{corollary}

\begin{proof}
  Evident by stability.
\end{proof}

\begin{corollary}\label{cor:missing6}
  In case (iii) the six sets $\{x_i,y_j,d\}$ of width 2, $x_i = b_i$ or
  $c_i$, $y_j \in F_j$ are missing from $\cH(i,j)$.
\end{corollary}

\begin{proof}
  By stability it is sufficient to prove $\{b_i,a_j,d\} \notin
  \cH(i,j)$.  This follows from $(1,b_j) \in \cH_1$ and $(a_i,c_j) \in
  \cH_2$ using $\nu(\cH(i,j))=2$.
\end{proof}

\begin{remark}
  There were 9 candidates both for $G \in \cH_2$ and also for sets of
  width 2 containing $d$ in $\cH(i,j)$.  We proved that not even half
  are actually in $\cH(i,j)$.  This will be of great help in proving
  Conjecture~\ref{conj:fbound}.
\end{remark}

\section{Some important special cases}\label{sec:cases}

We consider $\cH(R)$ for $R = (i_1,i_2,\dots,i_k)$.  To simplify
notation we set $F_{i_{ _\ell}} = \bigl\{a_1(\ell),\dots,a_k(\ell) \bigr\}$,
$\ell=1,\dots,k$.  $F_0 = (1,d_2,\dots,d_{k-1})$.

Let us define the partition~$T_1 \cup \cdots \cup T_k$ of $F_{i_1} \cup
\cdots \cup F_{i_k}$ by $T_q = \bigl\{ a_q(1),\dots,a_q(k) \bigr\}$.

\begin{definition}
  A set $D$ is called a \underline{partial diagonal} if
  $D \subset F_{i_1} \cup \cdots \cup F_{i_k},\ \nu(D) = |D|$ and
  $\bigl|D \cap T_q\bigr| \leq 1$ for all $1 \leq q \leq k$.  If further
  $|D| = k$, then it is called a \underline{diagonal}.
\end{definition}

\begin{definition}
  If a set $T$, $|T|=k$ satisfies $\bigl| T \cap F_{i_{ _\ell}} \bigr| = 1$
  for all $1 \leq \ell \leq k$, (or equivalently, $\nu(T) = k$) then $T$
  is called a \underline{transversal}.
\end{definition}

\begin{fact}
  There are $k^k$ transversals, $k!$ diagonals and for every diagonal
  $D$ there are $k!$ transversals $T$ satisfying $D \ll T$.
\end{fact}\qed

\begin{corollary}
  If there is a diagonal which is not in $\cH(R)$ then there are at
  least $k!$ transversals that are not in $\cH(R)$ either.
\end{corollary}\qed

\begin{definition}
  The $k$-tuple $R$ is called \underline{normal} if $1 \leq q < q' \leq
  k$ and $a \in T_q$, $a' \in T_{q'}$ imply $a < a'$.
\end{definition}

The notion of normality means that in $F_{i_1} \cup \cdots \cup
F_{i_k}$, the smallest elements are in $T_1$, the next smallest in $T_2$
and so on.  It is a rather strong property, which cannot be enforced in
general.  However, in some cases yes.

\begin{proposition}
  If all $k!$ diagonals are in $\cH(R)$, then $R$ is normal.
\end{proposition}

\begin{proof}
  Suppose for contradiction that for some $1 \leq q < q' \leq k$, $a \in
  T_q$, $a' \in T_{q'}$, $a > a'$ holds.

  Since $q \not= q'$, there exists a diagonal $D_1$ with
  $(a',a) \subset D_1$.  Take $(k-1)$ more diagonals $D_2,\dots,D_k$
  such that $D_1,D_2,\dots,D_k$ form a partition of
  $F_{i_1} \cup \cdots \cup F_{i_k}$.  Replace $F_{i_1},\dots,F_{i_k}$
  by $D_1,D_2,\dots,D_k$.  Should the elements of $D_i$ be listed in the
  order as in $F_{i_{ _\ell}}$, that is, the $h^{\text{th}}$ element is in
  $T_h$, then $\ds \sum_{1 \leq p \leq k}a_h(p)$ would be unchanged.
  However, they are reordered in increasing order.  The assumption
  $a > a'$ implies that some are really changed.  It is easy to see that
  the smallest $h$ for which there is a change in
  $\ds \sum_{1\leq p \leq h} a_h(p)$, it is decreasing.  That
  contradicts the minimal choice of $F_1,\dots,F_s$.
\end{proof}

\begin{definition}
  The $k$-tuple $R$ is called \underline{fat} if there exists pairwise
  disjoint $k$-sets $H_1,\dots,H_{k-1} \in \cH(R)$ such that $H_1 \cup
  \cdots \cup H_{k-1} = T_2 \cup \cdots \cup T_k$.
\end{definition}

This is also a very strong property.

\begin{proposition}
  If $\cH(R)$ is not fat then there are at least $(k-1)^{k-1}$
  transversals $T$ with $T \notin \cH(R)$.
\end{proposition}

\begin{proof}
  There are $(k-1)^k$ transversals in $T_2 \cup \cdots \cup T_k$.  It is
  easy to partition them into $(k-1)^{k-1}$ groups so that each group
  consists of $k-1$ transversals, forming a partition of $T_2 \cup
  \cdots \cup T_k$.  Since $\cH(R)$ is not fat, at least one transversal
  is missing from $\cH(R)$ for each group.
\end{proof}

The following lemma shows the strength of the above properties.

\begin{lemma}
  If $R$ is both fat and normal then $|H| = k$ holds for every $H \in
  \cH(R)$ with $H \subset F_{i_1} \cup \cdots \cup F_{i_k}$.
\end{lemma}

\begin{proof}
  Suppose that $H$ contradicts the conclusion.  Let $|H| = h < k$.
  Normality implies $(a_1(1),\dots,a_1(h)) \ll H$.  By stability,
  $(a_1(1),\dots,a_1(h)) \in \cH(R)$.

  From stability and Claim~\ref{cl:easy} we infer $\{
  a_1(k),d_1,d_2,\dots,d_{k-1}\} \in \cH(R)$.  Together with the $k-1$
  pairwise disjoint sets $H_1,\dots,H_{k-1}$ we obtain a contradiction
  with $\nu(\cH(R))=k$.
\end{proof}

\begin{remark}
  Using $d_1 = 1$, $F_0 \cup \{a_2(k)\} \in \cF$ follows from
  Claim~\ref{cl:easy}.  Therefore one can slightly relax the condition
  of fatness in the lemma and require only that $\bigl(T_2 \cup \cdots \cup
  T_k - \{a_2(k)\}\bigr) \cup \{a_1(k)\}$ can be obtained as the
  union of $(k-1)$ members of $\cH(R)$.
\end{remark}

\begin{definition}
  We say that $R$ is \underline{slightly fat} if there are $k-1$
  transversals $H_1,\dots,H_{k-1} \in \cH(R)$ whose union is $T_1 \cup
  T_3 \cup T_4 \cup \cdots \cup T_k$.
\end{definition}

One can prove in the above way

\begin{fact}
  If $R$ is slightly fat, $H \in \cH(R)$ then $H$ is not a proper subset
  of $T_2$.
\end{fact}\qed

Let us consider now $\cH(R)$ with plenty of $H \in \cH(R)$ with $|H| =
k-1$.

\begin{definition}
  We say that $R$ is \underline{robust} if there exist $k$ pairwise
  disjoint sets $H_1,\dots,H_k \in \cH(R)$, each of size $k-1$.
\end{definition}

\begin{claim}\label{cl:hr}
  If $H \in \cH(R)$ then $\bigl| H \cap (\{1\} \cup H_1 \cup \cdots \cup
  H_k) \bigr| \geq 2$ holds.
\end{claim}

\begin{proof}
  Suppose the contrary.  Then we can find $H_0$ with $H_0 \in \cH(R)$,
  $\bigl| H_0 \cap \bigl\{\{1\} \cup H_1 \cup \cdots \cup H_k\bigr\}
  \bigr| = 1$.  If $1 \in H_0$, then $H_0,H_1,\dots,H_k$ are $k+1$
  pairwise disjoint sets, contradicting $\nu(\cH(R))=k$.  However, if
  the intersection is some $x \in H_1 \cup \cdots \cup H_k$, then by
  stability $(H_0 \setminus \{x\}) \cup \{1\}$ is also in $\cH(R)$.
  Again we get $k+1$ pairwise disjoint sets.
\end{proof}

Let now $R$ be robust and $k=3$.  Set $X = F_{i_1} \cup F_{i_2} \cup
F_{i_3} \cup \{1,d\}$ and $Y = D_1 \cup D_2 \cup D_3 \cup \{1\}$.
Define $B(X,Y) = \bigl\{ F \in \binom{X}{3} : |F \cap Y| \geq 2 \bigr\}
\cup \binom{Y}{2}$.  Claim~\ref{cl:hr} implies that $\cH(R) \subseteq
B(X,Y)$.

Since $B(X,Y)$ corresponds to $\bigl\{ F \cap [3s+2] : F \in \cA_2(n)
\bigr\}$, $\ds \sum_{H \in \cH(R)}w(H) \leq f(2)$ holds almost automatically
for $\cH(R)$ if $R$ is robust.

\begin{claim}\label{cl:robust}
  For $k=3$, if $R$ is robust then
  $H_1 \cup \cdots \cup H_k = T_1 \cup T_2$ holds.
\end{claim}

\begin{proof}
  In the contrary case we can choose an $\ell$, $1 \leq \ell \leq k$ and
  an element $a \in (a_1(\ell),a_2(\ell))$ such that $a \notin H_1 \cup
  H_2 \cup \cdots \cup H_k$.

  Using Claim~\ref{cl:easy} and $d_1 = 1$, we infer
  $\{1,a,d_2,\dots,d_{k-1}\} \in \cF$.  Together with $H_1,\dots,H_k$
  these contradict $\nu(\cH(R)) = k$.
\end{proof}

\begin{proposition}
  For $k=3$, if $R$ is robust then
  \[
  \sum_{H \in \cH(R)}w(H) \leq f(2)
  \]
  holds.
\end{proposition}

\begin{proof}
  Since for $\cA_2(n)$ and all $R \in \binom{[s]}{k}$ one has
  $\cH_{\cA_2(n)}(R) = \bigl\{ H : |H \cap (T_1 \cup T_2 \cup \{1\})|
  \geq 2 \bigr\}$, Claims~\ref{cl:hr} and~\ref{cl:robust} imply $\cH(R)
  \subseteq \cH_{\cA_2(n)}(R)$ and the statement follows.
\end{proof}

Now let us prove a statement restricting the number of 2-sets in
$\cH(R)$ for the case that $R$ is not robust.  Let $g_2$ denote the
number of 2-element sets of width 2 in $\cH(R)$.  For $\{u,v\}$ let
$g(u,v)$ denote the number of $2$-element sets of width 2 in
$\cH(\{u,v\})$.  For $R = \{u,v,z\}$,
\begin{equation}
  \label{eq:obv}
  g_2 = g(u,v) + g(u,z) + g(v,z)
\end{equation}
is obvious.  For notational convenience we assume $g(u,v) \geq g(u,z)
\geq g(v,z)$.

\begin{proposition}\label{prop:63}
  If $R = \{u,v,z\}$ and $R$ is not robust then $g_2 \leq 9$ holds.
\end{proposition}

\begin{proof}
  For contradiction we assume $g_2 \geq 10$.  Using~\eqref{eq:obv} we
  distinguish two cases.

  \begin{itemize}
  \item[(a)] $g(u,v) = g(u,z) = 4$, $g(v,z) \geq 2$.

    In view of Proposition~\ref{prop:H2}, all four sets
    $\{a_u,b_v\}, \{a_v,b_u\}, \{a_u,b_z\}, \{a_z,b_u\}$ are in $\cH(R)$.
    Also, $g(v,z) \geq 2$ implies that either $\{a_v,b_z\}$ or
    $\{a_z,b_v\}$ is in $\cH(R)$.  By symmetry assume
    $\{a_v,b_z\} \in \cH(R)$.  Together with $\{a_z,b_u\}$ and
    $\{a_u,b_v\}$ these 3 sets show that $R$ is robust, a contradiction.

    \item[(b)] $g(u,v) = 4$, $g(u,z) = g(v,z) = 3$.

    If both $\{a_u,b_z\}$ and $\{a_z,b_u\}$ are in $\cH(R)$, the
    preceding proof works.  consequently, we may assume that for
    $\{u,z\}$ one has case (iii) in Proposition~\ref{prop:H2}.  That is,
    either $(a_u,c_z)$ or $(a_z,c_u)$ is in $\cH(R)$.  If
    $(a_u,c_z) \in \cH(R)$, take $\{b_u,b_v\}$ and $\{a_v,a_z\}$ to show
    that $R$ is robust.

    If $(a_z,c_u) \in \cH(R)$ then take $\{a_u,a_v\}$ and $\{b_u,b_v\}$
    to get the same contradiction.
  \end{itemize}

\end{proof}

\section{The case $s=2$}\label{sec:s2}

Let us use the results from Section~\ref{sec:struc} to show that the
Matching Conjecture is true for $s=2$.

Since in this case $\cA_3 = \binom{[8]}{3}$ has 56 elements and
$\cA_1(10) = \bigl\{ F \in \binom{[10]}{3} : F \cap [2] \not= \emptyset
\bigr\}$ has 64 elements, all we have to show is:
\[
|\cF| \leq 64 \quad \textrm{for} \ n=10, \ \cF \subset \binom{[n]}{3}, \
\nu(\cF) = 2.
\]
(Recall, that for $n=9=3(s+1)$, the bound $\binom{n-1}{3}$ is true for
all $s \geq 2$.)

As we showed before, $\nu(\cF(\bar{1}))=2$ can be assumed WLOG.  Now
$R=(1,2)$.  Define $\cG_i = \{ H \in \cH(1,2) : |H| = 2,\ \nu(H) = i \}$
for $i=1,2,\ \cG = \cG_1 \cup \cG_2$. Let us prove

\begin{proposition}
  If $\cF \subset \binom{[10]}{3}$ satisfies $\nu(\cF) = 2$,
  $\nu(\cF(\bar{1})) = 2$, then
  \begin{equation}
    \label{eq:small}
    |\cF| = \biggl| \cF \cap \binom{[8]}{3} \biggr| + 2 \bigl| \cG
    \bigr| \leq 63 
  \end{equation}
  holds.
\end{proposition}

\begin{proof}
  Set $g_i = \bigl| \cG_i \bigr|$ for $i=1,2$.  Suppose for
  contradiction that $|\cF| \geq 64$.  From~\eqref{eq:small} we infer
  $g_1 + g_2 \geq 4$.  In particular, $(1,a_1) \in \cG$.

  Now $\nu(\cF) = \nu(\cH(1,2)) = 2$ implies that $\cP
  \stackrel{\textrm{def}}{=} \{ H \in \cH(R) : H \subset ([8]-(1,a_1))
  \}$ is an intersecting family.  In particular, at least 10 of the 20
  subsets of size 3 in $\binom{[8]-(1,a_1)}{3}$ are missing from $\cF
  \cap \binom{[8]}{3}$.  Consequently, the first term on the RHS
  of~\eqref{eq:small} is at most 46, proving $g_1 + g_2 \geq 9$.

  Since not both~$(1,x_i)$ and $F_i - \{x_i\}$ are in $\cG$, for $i=1,2$,
  and $x_i \in F_i$ (cf. Fact~\ref{fact:notin},~\ref{fact:f0}), $g_1 \leq
  6$.  Consequently, $g_2 \geq 3$ follows.

  Now we can apply Proposition~\ref{prop:H2} and distinguish the
  following two cases
  \begin{itemize}
    \item[(a)]
      $(a_1,b_2)$ and $\{b_1,a_2\}$ are both in $\cH(1,2)$.

      \begin{claim}
        \begin{equation}
          \label{eq:fcap}
          | F \cap \{1, a_1, a_2, b_1, b_2 \}| \geq 2 \quad \textrm{for
            all} \ F \in \cF.
        \end{equation}
      \end{claim}

      Indeed, if $\bigl|F \cap \{1, a_1, a_2, b_1, b_2\} \bigr| \leq 1$ then by
      stability there exists some $F' \in \cF$ with $F' \cap
      \{a_1,a_2,b_1,b_2\} = \emptyset$.  Using $(a_1,b_2)$ and
      $\{b_1,a_2\}$ one concludes $\nu(\cH(1,2)) \geq 3$, a
      contradiction.

      The family $\cF$ of all $F \in \binom{[10]}{3}$
      satisfying~\eqref{eq:fcap} is exactly $\cF_2(10)$ and it has size
      \[
      \binom{5}{3} + 5\binom{5}{2} = 60 < 63 
      \]
    \item[(b)] $\cG_2 \cap \cH(1,2) = \bigl\{(a_1,a_2), (a_1,b_2),
      (a_1,c_2) \bigr\}$.

      Now $g_1 + g_2 \geq 9$ and $g_2 = 3$ imply $g_1 \geq 6$.  In
      particular $(1,b_1) \in \cG_1$ and one of $(1,c_1), (a_1,b_1)$ is
      in $\cG_1$ too.

      However, $(1,c_1) \in \cG_1$ implies $\{a_1,b_1,d\} \notin \cF$
      and $(a_1,b_1) \in \cG_1$ implies $\{1,c_1,d\} \notin \cF$.  In
      both cases we found a missing set from $\binom{[8]}{3}$ that is
      not contained in $[8] - (1,a_1)$.  Thus we proved $\bigl| \cF \cap
      \binom{[8]}{3} \bigr| \leq \binom{8}{3} - 10 - 1 = 45$.
      Now~\eqref{eq:small} and $g_1 + g_2 = 9$ imply
      \[
      |\cF| \leq 45 + 2\cdot 9 = 63
      \]
      as desired.
    \end{itemize}
\end{proof}

\section{Fat and sufficiently fat triples}\label{sec:fat}

Let us suppose that $R$ is a fat triple.  With notation $A =
(a_i,a_j,a_k)$, $B = \{b_i, b_j, b_k \}$, $C = \{c_i, c_j, c_k \}$ this
means that there are $F, F' \in \cH(R)$ with $F \cup F' = B \cup C$.

\begin{proposition}
  If $R$ is fat then~\eqref{eq:conj} holds.
\end{proposition}

\begin{proof}
  We claim that $\cG_2 \cap \cH(R) = \emptyset$.  Indeed, the contrary
  and stability would imply $(a_i, a_j) \in \cG_2$.  Since $\{1,a_k,d\}
  \in \cF$, together with $F$ and $F'$ we have 4 pairwise disjoint sets,
  a contradiction

  Comparing with $\cA_3$ we see that our maximum surplus is nine sets in
  $\cG_1$.  However, the existence of a set $(1,a_u)$ in $\cH(R)$,
  together with $F,F'$ imply that $\{a_v,a_z,d\} \notin \cH(R)$.  By
  stability, the 9 sets $\{x_v,x_z,d\}$, $x_v \in F_v,\, x_z \in F_z$
  are all missing.  Thus for a loss of a maximum of 3 sets
  ($(1,a_u),(1,b_u)$ and one of $(1,c_u), (a_u,b_u)$) we have a gain of
  9 sets of width 2.  Comparing weights (using
  Convention~\ref{conv:comp}),
  \[
  \begin{split}
    \frac{3(n-3s-3)}{\binom{s-1}{2}} &< \frac{9}{s-2} \quad \textrm{is
      equivalent to} \\
    2(n-3s-3) &< 3s-3, \quad \textrm{using $n-3s-3 \leq \frac{s}{2}$}.\\
    s &< 3s-3, \quad \textrm{true for $s \geq 3$}.
  \end{split}
  \]
\end{proof}

\begin{fact}\label{fact:notfat}
  If $R$ is not fat then in $\cH(R)$
  \begin{itemize}
  \item[(i)] at least 4 sets of width 3 are missing from $\binom{B \cup
      C}{3}$.

  \item[(ii)] at least 6 sets of width 2 are missing from
    $\binom{B \cup C}{3}$.
  \end{itemize}  
\end{fact}

\begin{proof}
  Let us look at the 10 unordered partitions of $B \cup C$ into 2 sets
  of size 3.  $(10 = \frac{1}{2}\binom{6}{3})$.  Since $R$ is not fat,
  at least one set from each pair is missing from $\cH(R)$.  Now 4
  partitions use sets of width 3, 6 use sets of width 2.
\end{proof}

We are going to compare $\cH(R)$ with $\cA_3$, that is, the complete
3-graph on the same 11 vertices.  Fact~\ref{fact:notfat} provides us
with a gain of $4 + \frac{6}{s-2}$.

\begin{proposition}
  If $(1,a_k) \notin \cH(R)$ then~\eqref{eq:conj} holds.
\end{proposition}

\begin{proof}
  First note that $(1,a_k) \notin \cH(R)$ implies $\bigl| \cG_1 \cap
  \cH(R) \bigr| \leq 6$ and $\cG_2 \cap \cH(u,k) = \emptyset$ for $u \in
  (i,j)$. Consequently, $\cG_2 \cap \cH(R) \subset \cH(i,j)$.  Set $g_2 =
  \bigl| \cG_2 \cap \cH(R) \bigr|$.

  Let us first prove~\eqref{eq:conj} for the case $g_2 \leq 2$.  Let $n
  = n_0(s,3)$.  Since we know $n_0(s,3) \leq 3.5s + 3$, we leave
  $n-3s-2 \leq \frac{s}{2} + 1 = \frac{s-2}{2} + 2$.

  Consequently our losses are at most
  \[
  2 \frac{\frac{s-2}{2}+2}{s-2} + 6\frac{\frac{s-1}{2} +
    \frac{1}{2}}{\binom{s-1}{2}} = 1 + \frac{10}{s-2} + \frac{6}{(s-1)(s-2)}.
  \]
  Let us compare it with our gains, $4 + \frac{6}{s-2}$.
  \begin{eqnarray}
    1 + \frac{10}{s-2} + \frac{6}{(s-1)(s-2)} &\leq& 4 + \frac{6}{s-2},
    \quad \textrm{equivalently,} \nonumber \\
    \frac{4}{s-2} + \frac{6}{(s-1)(s-2)} &\leq& 3. \label{eq:part}
  \end{eqnarray}

  For $s=4$ both sides are equal.  Since the LHS is a decreasing
  function of $s$,~\eqref{eq:part} holds for $s \geq 4$.  For $s=3$ we
  use $n_0(s,3)=13$, $n_0(s,3) - 3s - 2 = 2$ and check directly
  \[
  2 \cdot \frac{2}{1} + 6 \cdot \frac{1}{1} \leq 4 + \frac{6}{1}.
  \]

  Now let $g_2 \geq 3$.  Using Corollaries~\ref{cor:missing5}
  and~\ref{cor:missing6}, we get an extra gain of $\frac{5}{s-2}$.
  Moreover, if $g_2 = 4$, then $\{1,c_i,c_j\} \notin \cF$ (because
  $(a_i,a_j), \{b_i,b_j\} \in \cH(R)$).  Consequently,
  $\{a_u,c_i,c_j\} \notin \cF$ for $u \in R$.  These 4 sets provide us
  with an extra gain of $1 + \frac{3}{s-2}$.

  Thus the inequalities to check in the two cases are:
  \[
  \begin{split}
    \frac{3}{2} + \frac{12}{s-2} + \frac{6}{(s-1)(s-2)} &\leq 4 +
    \frac{11}{s-2} \quad (g_2 = 3) \\
    2 + \frac{14}{s-2} + \frac{6}{(s-1)(s-2)} &\leq 5 + \frac{14}{s-2}
    \quad (g_2 = 4)
  \end{split}
  \]
  Rearranging gives
  \[
  \begin{split}
    \frac{1}{s-2} + \frac{6}{(s-1)(s-2)} &\leq \frac{5}{2}\\
    \frac{6}{(s-1)(s-2)} &\leq 3.
  \end{split}
  \]
  The first holds for $s \geq 4$, the second for $s \geq 3$.  If $g_2 =
  3$ and $s=3$ then using $n-3s-2=2$ one checks directly
  \[
  3 \cdot \frac{2}{1} + 6 \cdot \frac{1}{2} < 4 + \frac{11}{1}
  \]
\end{proof}

From now on $R$ is not fat and $(1,a_u) \in \cH(R)$ for all $u \in R$.

For a non-fat triple $R$ some slightly weaker properties might hold.

\begin{definition}
  We measure the fatness of $R$ by the set $Q \subseteq (i,j,k)$ by
  defining $Q = Q(R)$ through: $u \in Q$ if and only if there exist
  pairwise disjoint $F,F' \in \cF$ with $F \cup F' =
  \{a_u,b_v,b_z,c_u,c_v,c_z\}$. If $Q \not= \emptyset$, $R$ is called
  \underline{sufficiently fat}.
\end{definition}

\begin{proposition}\label{prop:QH}
  If $u \in Q$ then $\{a_v,a_z\} \notin \cH(R)$.
\end{proposition}

\begin{proof}
  It follows from $\nu(\cH(R)) = 3$ since the 4 sets $F,F',\{1,b_u,d\}$
  and $\{a_v,a_z\}$ are pairwise disjoint.
\end{proof}

\begin{corollary}
  ~
  \begin{itemize}
  \item[(i)] If $|Q|=3$ then $(a_i,a_j) \notin \cH(R)$.

  \item[(ii)] If $|Q|=2$ then $(a_i,a_k) \notin \cH(R)$.

  \item[(iii)] If $|Q|=1$ then $(a_j,a_k) \notin \cH(R)$ hold.
  \end{itemize}
\end{corollary}

\begin{proof}
  Immediate from Proposition~\ref{prop:QH} and $(a_i,a_j) \ll (a_i,a_k)
  \ll (a_j,a_k)$.
\end{proof}

Define
$\cF_\ell = \{ F \in \cF_0 : \bigl| F \bigr| = 3,\, v(F) = \ell \},\ \ell
= 2,3$.
Define further $\cT = \cF_3 \cap \cH(R)$.  Let us show that, for not
sufficiently fat triples, $\cT$ is relatively small.

\begin{proposition}\label{prop:84}
  Suppose that $R = (i,j,k)$ is not sufficiently fat.  Then
  \begin{itemize}
  \item[(i)] $|T| \leq 20$, and

  \item[(ii)] there are at least 12 missing edges from $\cF_2 \cap \cH(R)$.
  \end{itemize}
\end{proposition}
\begin{proof}
  (i) There are 8 transversals in $U \stackrel{\text{def}}{=} B \cup C$.
  If $\{b_i,b_j,b_k\}$ is missing then by stability all are missing.
  The next smallest in the shifting partial order are $\{b_u,b_v,c_z\}$,
  $z \in (i,j,k) : (u,v) = (i,j,k) - \{z\}$.  Supposing indirectly
  $|T| \geq 21$, we may assume that, for one fixed $z$,
  $\{b_u,b_v,c_z\} \in \cT$ holds.

  Since $(i,j,k)$ is not sufficiently fat, $\{c_u,c_v,a_z\} \notin \cT$.
  Consider two more similar 3-sets: $\{c_u,a_v,c_z\}$ and
  $\{a_u,c_v,c_z\}$.  If both are missing from $\cT$, then by stability
  we obtain 7 missing sets and $|T| \leq 27-7 = 20$.  Thus one or both
  are in $\cT$.  We distinguish two cases accordingly.
  \begin{itemize}
  \item[(a)] $\{c_u,a_v,c_z\}, \{a_u,c_v,c_z\} \in \cF$

    Since $(i,j,k)$ is not sufficiently fat, neither $\{b_u,c_v,b_z\}$
    nor $\{c_u,b_v,b_z\}$ are in $\cF$.  By stability, out of the 8
    transversals of $U$, only $\{b_u,b_v,b_z\}$ and $\{b_u,b_v,c_z\}$
    are in $\cF$.  Together with $\{c_u,c_v,a_z\}$, we have 7 missing
    sets proving $\bigl|\cF_3 \cap R(i,j,k)\bigr| \leq 20$.

  \item[(b)] $\{c_u,a_v,c_z\} \notin \cF$, $\{a_u,c_v,c_z\} \in \cF$.

    Now $\{a_u,c_v,c_z\} \in \cF$ implies $\{c_u,b_v,b_z\} \notin \cF$.
    Thus by stability, $\{c_u,x_u,x_z\} \notin \cF$ for
    $x_v \in (b_v,c_v)$, $x_z \in (b_z,c_z)$.  Together with
    $\{c_u,c_v,a_z\}$ and $\{c_u,a_v,a_z\}$ these are already 6 missing
    sets.  If no more are missing, $\{b_u,c_v,c_z\}$ and
    $\{c_u,b_v,a_z\}$ would be in $\cF$.  However that would show that
    $(i,j,k)$ is quite fat, a contradiction.
  \end{itemize}

  (ii) Consider the following 12 disjoint pairs.
  \[
  \begin{split}
    &\{b_u,c_u,c_v\}, \ \{b_v,a_z,c_z\} \quad \textrm{and}\\
    &\{b_u,c_u,b_v\}, \ \{c_v,a_z,c_z\}, \quad
    \textrm{$u,v,z$ is a permutation of $(i,j,k)$}
  \end{split}
  \]
  Since $(i,j,k)$ is not sufficiently fat, at least one set of each pair
  is missing.  These are distinct sets of width 2, concluding the proof.
\end{proof}

Even if $R$ is sufficiently fat, but $|Q|=1$, we can prove bounds
slightly worse than (i) and (ii).

\begin{proposition}\label{prop:85}
  If $|Q|=1$ then (i), (ii) hold.
  \begin{itemize}
  \item[(i)] $|\cF| \leq 21$.

  \item[(ii)] There are at least 10 missing edges from $\cF_2 \cap \cH(R)$.
  \end{itemize}
\end{proposition}

\begin{proof}
  Let $Q=\{z\}$.  Let us define the two six element sets $P(x) = (B \cup
  C - \{b_x\}) \cup \{a_x\}, \, x=u,v$.  By the definition of $Q =
  Q(R)$, there are no $F,F' \in \cF$ with $F \cup F' = P(x)$.  Therefore
  -- just as in the proof of Fact~\ref{fact:notfat} -- if $F \cup F' =
  P(x)$ is a partition of $P(x)$, then at least one of $F,F'$  is not in
  $\cH(R)$.

  Let us list the 4 partitions of $P(u)$ into sets of width 3:
  \[
  \begin{split}
    &\{a_u,b_v,b_z\}, \ \{c_u,c_v,c_z\}\\
    &\{a_u,b_v,c_z\}, \ \{c_u,c_v,b_z\}\\
    &\{a_u,c_v,b_z\}, \ \{c_u,b_v,c_z\}\\
    &\{a_u,c_v,c_z\} \  \{c_u,b_v,b_z\}
  \end{split}
  \]

  Let us list further 2 of the partitions of $P(v)$ into 2 sets of width
  3:
  \[
  \begin{split}
    &\{c_u,a_v,b_z\}, \ \{b_u,c_v,c_z\}\\
    &\{c_u,a_v,c_z\}, \ \{b_u,c_v,b_z\}
  \end{split}
  \]
  These are altogether 6 partitions using 12 distinct sets, proving (i).

  To prove (ii), we make the corresponding list of 10 partitions into
  sets of width 2.
  \[
  \begin{split}
    &\{a_u,c_u,b_v\}, \ \{c_v,b_z,c_z\}\\
    &\{a_u,c_u,c_v\}, \ \{b_v,b_z,c_z\}\\
    &\{a_u,c_u,b_z\}, \ \{b_v,c_v,c_z\}\\
    &\{a_u,c_u,c_z\}, \ \{b_v,c_v,b_z\}\\
    &\{a_u,b_v,c_v\}, \ \{c_u,b_z,c_z\}\\
    &\{a_u,b_z,c_z\}, \ \{c_u,b_v,c_v\}\\
    \cline{1-2}
    &\{b_u,c_u,c_v\}, \ \{a_v,b_z,c_z\}\\
    &\{b_u,c_u,b_z\}, \ \{a_v,c_v,c_z\}\\
    &\{b_u,c_u,c_z\}, \ \{a_v,c_v,b_z\}\\
    &\{b_u,b_z,c_z\}, \ \{c_u,a_v,c_v\}
  \end{split}
  \]
\end{proof}

\begin{remark}
  The proof might look like trial and error, but it is not.  There is
  the underlying idea that $P(u)-P(v)=\{a_u,b_v\}$.  Thus is $F \cup F'
  = P(u)$ is a partition with $a_u \in F$, $b_v \in F'$ then neither
  $F$, nor $F'$ is a subset of $P(v)$.  This also implies that in case
  of equality in (i) or (ii) for those partitions where $F$ contains
  both $a_u$ and $b_v$, $F \in \cH(R)$, $F' \notin \cH(R)$ must hold.
\end{remark}

\section{Sufficiently fat is sufficient}\label{sec:sfat}

Let us prove~\eqref{eq:conj} with $\cA_3$ as a reference for triples $R$
that are sufficiently fat.  We distinguish cases according to $|Q|$.

Recall the notation $g_\ell = |\cG_\ell \cap \cH(R)|, \, \ell = 1,2$.
Our maximal losses can be estimated from above as
\begin{equation}
  \label{eq:loss}
  \frac{g_2\lfloor \frac{s+2}{2} \rfloor}{s-2} + \frac{g_1 \lfloor
    \frac{s}{2} \rfloor}{\binom{s-1}{2}}
\end{equation}

As to our gains, since $R$ is not fat, we have at least
\begin{eqnarray}
  \label{eq:92} 4 &+& \frac{6}{s-2} \quad (|Q| \geq 2), \ \text{and}\\
  6 \label{eq:93} &+& \frac{10}{s-2} \quad (|Q| = 1).
\end{eqnarray}
These are the ``basic'' gains.  That is, we can use
Corollaries~\ref{cor:missing5} and~\ref{cor:missing6} for some
additional gains in case that $|\cG_2 \cap \cH(u,v)| \geq 3$.

\begin{proposition}
  If $|Q| = 3$ then \eqref{eq:conj} holds.
\end{proposition}

\begin{proof}
  In view of Proposition~\ref{prop:QH}, $g_2 = 0$.  Thus we have to
  prove
  \begin{equation}
    \label{eq:94}
    \frac{9 \lfloor \frac{s}{2} \rfloor}{\binom{s-1}{2}} \leq  4 +
    \frac{6}{s-2}.
  \end{equation}
  For $s=3$, it is true.  Let $s \geq 4$ and use $\lfloor \frac{s}{2}
  \rfloor \leq \frac{s-1}{2} + \frac{1}{2}$.  Then~\eqref{eq:94} reduces
  to
  \[
  \frac{3}{s-2} + \frac{9}{(s-1)(s-2)} \leq 4.
  \]
  For $s=4$, we have $3 < 4$, and the LHS is a decreasing function of $s$.
\end{proof}

\begin{proposition}
  If $|Q|=2$, then~\eqref{eq:conj} holds unless $s=3, \ n = n_0(3,3) = 13$.
\end{proposition}

\begin{proof}
  Stability and Proposition~\ref{prop:QH} imply $(a_i,a_k), (a_j,a_k)
  \notin \cG$.  Thus $g_2 = |\cG_2 \cap \cH(i,j)|$.  We distinguish 2 cases
  accordingly $g_2 \leq 2$ and $g_2=3$ or $4$.
  \begin{itemize}
  \item[(a)] $g_2 \leq 2$

    First let $s \geq 6$.  Use $\frac{s+2}{2} = \frac{s-2}{2} + 2$ to
    get the upper bound for~\eqref{eq:loss}:
    \[
    2 \frac{\frac{s-2}{2}+2}{s-2} + \frac{9}{s-2} + \frac{9}{(s-1)(s-2)}
    = 1 + \frac{13}{s-2} + \frac{9}{(s-1)(s-2)}
    \]
    Thus it is sufficient to have
    \[
    \frac{7}{s-2} + \frac{9}{(s-1)(s-2)} \leq 3.
    \]
    For $s=6$, $\frac{7}{4} + \frac{9}{20} < 3$, and the LHS is monotone
    decreasing with $s$.

    For $s=5$, $\lfloor \frac{s+2}{2} \rfloor = 3, \ \lfloor \frac{s}{2}
    \rfloor = 2$ and
    \[
    \frac{2 \cdot 3}{3} + \frac{9 \cdot 2}{6} = 5 < 4 + \frac{6}{3}
    \quad \text{holds}.
    \]
    For the cases $s=3$ or $4$, let first $n = n_0(s,3) - 1$.  Then
    $n-3s-2$ is $1$ for $s=3$ and $2$ for $s=4$.  It can be checked
    that~\eqref{eq:loss} is less than~\eqref{eq:92} in both cases.

    For $s=4, \ n = n_0(4,3) = 17$ one has $|\cA_1(17)| - \binom{14}{3}
    = 30$.  Thus it is sufficient to prove (using $f(1) = f(3) +
    \frac{30}{\binom{4}{3}}$) that~\eqref{eq:loss} is less
    than~\eqref{eq:92} plus 7.5, which holds largely.  However, for
    $s=3, \ n=n_0(3,3)=13$ one has
    \[
    2 \cdot 3 + 9 = 15 > 4 + 6.
    \]
    We shall take care of the $s=5,\, n=13$ case separately in
    Section~\ref{sec:last}.
    
  \item[(b)] $g_2 \geq 3$.

    From Proposition~\ref{prop:H2} it follows that $g_2 = 3$ or $4$.
    From Corollaries~\ref{cor:missing5} and~\ref{cor:missing6} we can
    replace \eqref{eq:92} by $4 + \frac{11}{s-2}$.  Moreover, in the
    case $g_2 = 4$, $\{1,c_i,c_j\} \notin \cF$ and stability provide us
    with $4$ previously not excluded missing sets $\{1,c_i,c_j\},
    \{a_i,c_i,c_j\}, \{a_j,c_i,c_j\}$ and $ \{a_k,c_i,c_j\}$.  Among them 3
    are of width 2 and 1 is of width 3, providing for an extra gain of
    $1 + \frac{3}{s-2}$.

    Consequently, the inequalities needed for $g_2=3,4$ are the
    following.
    \begin{eqnarray}
      \frac{3\cdot\frac{s+2}{2}}{s-2} + \frac{9\cdot\frac{s}{2}}{\binom{s-1}{2}}
      &\leq&  4 + \frac{11}{s-2}, \quad \text{and}
      \label{eq:95}\\
      \frac{4\cdot\frac{s+2}{2}}{s-2} + \frac{9\cdot\frac{s}{2}}{\binom{s-1}{2}}
      &\leq& 5 + \frac{14}{s-2} \nonumber
    \end{eqnarray}
    The second one holds with equality for $s=4$.  The first one holds
    strictly for $s=5$.  Collecting the terms with $\frac{1}{s-2}$ on
    the LHS and using monotonicity, both inequalities follow unless
    $s=4$ in the first one.  However, even in this case the LHS is only
    1 larger than the RHS.  Consequently,~\eqref{eq:conj} holds easily
    with $f(3)$ replaced by $f(1)=f(3)+7.5$.  In the case $s=4$,
    $n=n_0(4,3)-1 = 16$, instead of~\eqref{eq:95} we need
    \[
    \frac{3 \cdot \frac{4}{2}}{2} + \frac{9\cdot\frac{2}{2}}{3} = 3+3 < 4 +
    \frac{11}{2}
    \]
    which is true by large
  \end{itemize}
\end{proof}

\begin{proposition}\label{prop:93}
  \eqref{eq:conj} holds for $|Q|=1$. $(s \geq 4)$.
\end{proposition}

\begin{proof}
  In view of Proposition~\ref{prop:QH}, $(a_j,a_k) \notin \cG$.  Thus
  \begin{equation}
    \label{eq:96}
    g_2 = |\cG_2 \cap \cH(i,j)| + |\cG_2 \cap \cH(i,k)|.
  \end{equation}
  Using Proposition~\ref{prop:85} provides us with a gain of $6 +
  \frac{10}{s-2}$.

  \begin{claim}\label{cl:91}
    For $s \geq 5$ one has
    \begin{equation}
      \label{eq:97}
      5 \cdot \frac{\lfloor \frac{s+2}{2} \rfloor}{s-2} + \frac{9 \cdot
        \lfloor \frac{s}{2} \rfloor}{\binom{s-1}{2}} \leq 6 + \frac{10}{s-2}
    \end{equation}
  \end{claim}
  \begin{proof}
    \eqref{eq:97} is easily checked to hold for both $s=5$ and $6$.  For
    $s > 6$ monotonicity considerations yield~\eqref{eq:97}.

    For $s=4$ the LHS of~\eqref{eq:97} is $\frac{15}{2} + 6 = 13.5$, the
    RHS is 10.  Since the difference is less than $7.5$, we are
    alright.

    In the case $s=4, \ n=16$ one can replace $\frac{s+2}{2}$ by
    $\frac{s}{2}, \ \frac{s}{2}$ by $\frac{s-2}{2}$ and the
    corresponding version of~\eqref{eq:97} holds in the stronger form
    \[
    8 \cdot \frac{2}{2} + \frac{9}{3} \leq 6 + \frac{10}{2},
    \]
    that is for $g_2=8$.  Consequently, in the sequel we do not need to
    consider the case $s=4, \ n=16$.

    In view of Claim~\ref{cl:91}, we can assume $g_2 \geq 6$.  Let us
    use~\eqref{eq:96}.  For $g_2 = 8, \ |\cG_2 \cap \cH(i,j)| = |\cG_2
    \cap \cH(i,k)| = 4$.  For $g_2 = 7$, one of them is $4$, the other
    is $3$.  For $g_2=6$, $6 = 4+2$, or $6 = 3+3$ hold.

    Let us first check the case $g_2 = 6$.  Now
    Corollaries~\ref{cor:missing5},~\ref{cor:missing6} provide us with
    an extra gain of $\frac{5}{s-2}$.  Thus we need
    \[
    \frac{6 \cdot \lfloor \frac{s+2}{2} \rfloor}{s-2} + \frac{9 \cdot
      \lfloor \frac{s}{2} \rfloor}{\binom{s-1}{2}} \leq 6 + \frac{15}{s-2}
    \]
    This inequality is true for both $s=5$ and $6$.  By monotonicity it
    holds for all $s \geq 5$.  For $s=4$ the two sides are $15$ and
    $13.5$ showing that the extra 7.5 is more than sufficient.

    In the cases of $g_2=7,8$ we can use the extra gains from
    Corollaries~\ref{cor:missing5},~\ref{cor:missing6}.  These amount to
    $\frac{10}{s-2}$, for missing sets containing $d$.  For the extra
    gains from Fact~\ref{fact:51}, that is the 4 sets
    $\{1,c_i,c_x\}, \{a_i,c_i,c_x\}, \{a_x,c_i,c_x\}$ and
    $\{a_y,c_i,c_x\}$, where $x=j$ or $k$ and $\{y\} = \{j,k\}-\{x\}$,
    we have to be more careful to avoid counting the same missing set
    twice.  The problem is coming from the fact that we are already
    using Proposition~\ref{prop:85}.  The sets containing 1 are safe,
    there is no such set in Proposition~\ref{prop:85}.

    Let us sort it out a little.  Note that from
    Proposition~\ref{prop:QH} we infer $Q = \{i\}$.  That is, the $u,v$
    in Proposition~\ref{prop:85} are $j$ and $k$.  Consequently, the
    sets containing $a_i$ do not occur there either.  Thus along with
    $\{1,c_i,c_j\}, \{1,c_i,c_k\}$, the two sets $\{a_i,c_i,c_k\}$ and
    $\{a_i,c_i,c_j\}$ provide us with extra gains of $\frac{4}{s-2}$.
    However the same cannot be said about the other sets.  For our
    purpose it is enough already.  We have now gains of
    \[
    6 + \frac{10}{s-2} + \frac{2 \cdot 5}{s-2} + \frac{4}{s-2} = 6 +
    \frac{24}{s-2}.
    \]
  \end{proof}

  \begin{claim}
    For $s \geq 4$
    \[
    \frac{8\cdot \lfloor \frac{s+2}{2} \rfloor}{s-2} + \frac{9 \cdot
      \lfloor \frac{s}{2} \rfloor}{\binom{s-1}{2}} \leq 6 + \frac{24}{s-2}.
    \]
  \end{claim}
  \begin{proof}
    For $s=5$ we have
    \[
    \frac{8 \cdot 3}{3} + \frac{9 \cdot 2}{6} = 11 < 6 + \frac{24}{3} = 14.
    \]
    For $s=4$ we have
    \[
    \frac{8 \cdot 3}{2} + \frac{9 \cdot 2}{3} = 18 = 6 + \frac{24}{2}.
    \]
    The rest follows from monotonicity.
  \end{proof}
  This concludes the proof of Proposition~\ref{prop:93}.
\end{proof}

\section{Not sufficiently fat is sufficient}\label{sec:notfat}

In view of Section~\ref{sec:sfat}, we may suppose that $R$ is not
sufficiently fat.  By Proposition~\ref{prop:84} we leave an initial gain
of
\begin{equation}
  \label{eq:101}
   7 + \frac{12}{s-2}.
\end{equation}

For each $(u,v) \subset R$ satisfying $|\cG_2 \cap \cH(u,v)| \geq 3$ we
have an additional gain of $\frac{5}{s-2}$.  Moreover, if $|\cG_2 \cap
\cH(u,v)| = 4$, then we can add to this $\frac{1}{s-2}$ for the missing
set $\{1,c_i,c_j\}$.

Let us compare our maximal loss with~\eqref{eq:101}
\begin{equation}
  \label{eq:102}
  \frac{g_2 \lfloor \frac{s+2}{2} \rfloor}{s-2} + \frac{9 \cdot \lfloor
    \frac{s}{2} \rfloor}{\binom{s-1}{2}} \leq 7 + \frac{12}{s-2}.
\end{equation}

For $s=5$ we have
\[
g_2 + 3 \leq 7 + 4
\]
which is true even for $g_2 = 8$.  For $g_2 = 9$, that is, increasing
$g_2$ by 1, increases the LHS by 1.  However, adding $\frac{5}{s-2}$ to
the RHS, it increases by $\frac{5}{3}$, proving~\eqref{eq:conj} for
$s=5$.

For $s \geq 6$ we use $\frac{s+2}{2} = \frac{s-2}{2} + 2, \ \frac{s}{2}
= \frac{s-1}{2} + \frac{1}{2}$ to rewrite the LHS of~\eqref{eq:102} as
\[
\frac{g_2}{2} + \frac{2g_2}{s-2} + \frac{9}{s-2} + \frac{9}{(s-1)(s-2)}
\]
and use it to rewrite~\eqref{eq:102} as
\begin{equation}
  \label{eq:103}
  \frac{2g_2-3}{s-2} + \frac{9}{(s-1)(s-2)} \leq 7 - \frac{g_2}{2}.
\end{equation}
In this form, for $g_2$ fixed, the RHS is constant and the LHS is a
decreasing function of $s$.  If it holds for $s=6$, it holds for all $s
\geq 6$.  For $g_2 = 6$, the inequality~\eqref{eq:103} reduces to
\[
\frac{9}{4} + \frac{9}{20} \leq 4,
\]
which is true.

For $g_2 \geq 7$, at least one $\cG_2 \cap \cH(u,v)$ has to contain at
least 3 elements.  Thus our gains increase by $\frac{5}{s-2}$ leading to
the adjusted version of~\eqref{eq:103}:
\[
\frac{2g_2-8}{s-2} + \frac{9}{(s-1)(s-2)} \leq 7 - \frac{g_2}{2}.
\]

For $g_2 = 8$, plugging in $s=6$ gives
\[
2 + \frac{9}{20} \leq 3
\]
which is true, and the case $s \geq 6$ follows by monotonicity.

For $g_2=9, \ 9 > 4+2+2$ implies that we can add $2 \cdot \frac{5}{s-2}$
to increase our gains.  Consequently, the inequality that we have to
prove reduces to
\[
\frac{2g_2-13}{s-2} + \frac{9}{(s-1)(s-2)} \leq 7 - \frac{g_2}{2}.
\]
Plugging in $g_2 = 9, \ s=6$ gives
\[
\frac{5}{4} + \frac{9}{20} \leq \frac{5}{2}
\]
which is true.  Thus we have proved the next proposition except for
$s=4$.

\begin{proposition}
  If $R$ is not sufficiently fat and $g_2 \leq 9$ then~\eqref{eq:conj}
  holds for $s \geq 4$.
\end{proposition}

\begin{proof}
  We only have to deal with the case of $s=4$.  There are 2 sub-cases:
  $n=16$ and $n=17$.  In the first case our losses can be written as
  \[
  g_2 + \frac{9}{6} \leq 10.5 < 7 + \frac{12}{2} = 13.
  \]
  For the case $n=17,\ n-3s-2=\frac{s+2}{2}$.  We can bound our losses
  as:
  \begin{equation}
    \label{eq:104}
    \frac{3g_2}{2}+6
  \end{equation}
  Since our gains are $7+\frac{12}{s-2} = 13$, we need only
  that~\eqref{eq:104} is less than $20.5$.  Fortunately, even for $g_2
  = 9$ one has
  \[
  \frac{3g_2}{2} + 6 = \frac{27}{2} + 6 = 19.5
  \]
  concluding the proof
\end{proof}

At this stage our proof is complete except for $s=3, \ n=n_0(3,3)=13$.
We are going to handle this case directly in Section~\ref{sec:last}.
One might think that our whole proof, which in its initial parts used
induction, might collapse without this case.  It is not the case.
Applying induction for some particular $s$, we always have $n \geq
n_0(s,3)-1 \geq n_0(s-1,3)+2$.  Therefore, to support the induction, it
is sufficient to prove that the maximum size of a 3-graph on
$n=n_0(s-1,3)+2$ vertices is at most $|\cA_1(n)|$.  In particular, in
our ``missing'' case, $n=16,\ s-1=3$, using $\binom{s-1}{3} = 1$ we nee
to give a bound of the form
\[
\sum_{H \in \cH(R)}w(H) \leq |\cA_1(15)| = \binom{15}{3} - \binom{12}{3}
= 235 = |\cA_3| + 70.
\]
That is, we do not have to struggle to get $f(3)$ or $f(3)+1$ as an
upper bound, $f(3)+70$ is sufficient.  That is too easy, the bounds we
have proven so far are much stronger.

\section{The last case}\label{sec:last}

Let $n=13,\ s=3,\ \cF \subseteq \binom{[13]}{3},\ \nu(\cF)=3$.  Since
for $s=3,\ s-2 = 1 = \binom{s-1}{2}$, computation is easier.  With
previous notation let $2 \leq d \leq 11$ and let
\[
F_1 \cup F_2 \cup F_3 = [11] - (1,d), \quad \text{where}\ F_i = (a_i,b_i,c_i).
\]

Set $\cG_i = \Bigl\{ G \in \binom{[11]}{2} : v(G)=i,\ \exists F \in \cF
: F \cap [11] = G \Bigr\}$, and $g_i = |\cG_i|$ for $i=1,2$.  Set
further $\cG = \cG_1 \cup \cG_2$ and $\cF_1 = \{ F \in \cF : F \subset
[11] \}$.  Now the formula for $|\cF|$ is simple
\begin{equation}
  \label{eq:111}
  |\cF| = |\cF_1| + 2|\cG| = |\cF_1| + 2(g_0 + g_1).
\end{equation}

\begin{proposition}\label{prop:111}
  $\displaystyle |\cF| \leq |\cA_3| = \binom{11}{3} = 165$.
\end{proposition}

Arguing indirectly we assume $|\cF| \geq 166 = |\cA_1(13)|$.  We are
going to prove Proposition~\ref{prop:111} as an end result of a series
of claims.

\begin{claim}\label{cl:111}
  $(1,2) \in \cG$.  
\end{claim}

\begin{proof}
  Otherwise $|\cG| = 0$ by stability and~\eqref{eq:111} implies $|\cF|
  \leq 165$.
\end{proof}

\begin{claim}
  \label{cl:112}
  \begin{equation}
    \label{eq:112}
    |\cF_1| \leq \binom{11}{3} - \binom{8}{2} = 137.
  \end{equation}
\end{claim}

\begin{proof}
  Consider $\tilde{\cF} \stackrel{\text{def}}{=} \{F \in \cF_1 : F
  \subset [3,11]\}$.  Now $\nu(\tilde{\cF}) \leq 2$ follows from $(1,2)
  \in \cG$.  Since $\bigl|[3,11]\bigr| = 9$, from the $s=2$ case we infer
  $|\tilde{\cF}| \leq \binom{8}{3} = \binom{9}{3} - \binom{8}{2}$.  That
  is, we showed that at least $\binom{8}{2}$ sets are missing already on
  $\binom{[3,11]}{3}$.  It can not be less on $\binom{[11]}{3}$,
  proving~\eqref{eq:112}.
\end{proof}

\begin{corollary}
  \label{cor:111}
  $g_1 + g_2 \geq 15$.
\end{corollary}

\begin{proof}
  If $g_1 + g_2 \leq 14$ then combining it with~\eqref{eq:112} and
  using~\eqref{eq:111} gives
  \[
  |\cF| \leq 137 + 2 \cdot 14 = 165
  \]
\end{proof}

In Section~\ref{sec:cases} we proved Conjecture~\ref{conj:fbound} for
robust triples.  Since we are arguing indirectly, WLOG $[3]$ is not
robust.  Thus Proposition~\ref{prop:63} gives $g_2 \leq 9$.  We showed
also (the much easier) inequality $g_1 \leq 9$.  Along the lines of
Proposition~\ref{prop:63} let us prove:

\begin{claim}
  \label{cl:113}
  \begin{equation}
    \label{eq:113}
    g_1 + g_2 \leq 17
  \end{equation}
\end{claim}

\begin{proof}
  Arguing indirectly we assume $g_1 = g_2 = 9$.  For $(u,v) \subset [3]$
  let $\cG(u,v)$ denote the family of those $G \in \cG_2$ that satisfy
  $G \subset F_u \cup F_v$.  In Proposition~\ref{prop:H2} we
  characterized $\cG(u,v)$ for $|\cG(u,v)| \geq 3$.  Let us show that
  possibilities (i) and (iii) cannot occur simultaneously.  Indeed if
  $|\cG(u,v)| = 4$ for some $\{u,v\} \subset [3]$, and either
  $(a_u,c_z)$ or $(a_z,c_u)$ is in $\cG$, then we can take $(a_u,c_z),
  \{b_u,b_v\}$ and $\{a_v,a_z\}$ or the 3 sets $(a_z,c_u), \{a_u,a_v\},
  \{b_u,b_v\}$ to show that $[3]$ is robust, a contradiction.

  Should no $(a_u,c_z)$ be in $\cG$, then there are only $3 \cdot 4 =
  12$ possibilities for $G \in \cG_2$.  These $12$ sets can be
  partitioned into $4$ groups of $3$ sets each, where each group gives a
  partition of $A \cup B$.  Since $[3]$ is not a robust triple, at most
  $2$ sets from each group are in $\cG_2$.  Thus $|\cG_2| \leq 4 \cdot 2
  = 8 < 9$.

  Until now we showed that there is at least one $(u,v)$ with $(a_u,c_v)
  \in \cG$, there is no $(u,v)$ with $|\cG(u,v)| = 4$.  Hence by $g_2 =
  9$, $|\cG(u,v)| = 3$ for each $(u,v) \subset [3]$.

  Let us show that possibility (ii) cannot hold for two choices of
  $(u,v) \subset [3]$.  Indeed, if it held for, say, $\{u,z\}$ and
  $\{v,z\}$ and $(a_u,c_v) \in \cG$, then we could use $(a_u,c_v)$,
  $\{b_u,a_z\}$ and $\{b_z,a_v\}$ to show that $[3]$ is robust.

  Note that if $(a_2,c_3) \in \cG$ then by stability $(a_1,c_3) \in \cG$
  holds as well.  Consequently, we are left with only two possibilities.
  \begin{itemize}
  \item[(a)] $(a_1,c_2) \in \cG,\ (a_1,c_3) \in \cG,\ \cG(2,3)$ is of
    type (ii).

  \item[(b)] $(a_1,c_2),\, (a_1,c_3),\, (a_2,c_3) \in \cG$.
  \end{itemize}
  Let us consider these separately.
  \begin{itemize}
  \item[(a)] $(a_1,c_2),\, (a_2,b_3),\, \{a_3,b_2\}$ show that $[3]$ is
    robust.

  \item[(b)]
    In this case we are going to prove $|\cF| \leq 165$.

    First let us show that $\cF_1 \cap \binom{F_3 \cup \{b_2,c_2\}}{3} =
    \{F_3\}$.  Let $H \subset (F_3 \cup \{b_2,c_2\})$ and $H \not= F_3$
    satisfy $H \in \cF_1$.  By stability, we may assume that either $H =
    \{b_2,c_2,a_3\},$ or $H = \{b_2,a_3,b_3\}$.

    In the first case look at the $4$ sets $\{b_2,c_2,a_3\}$,
    $(a_1,c_3),\, (a_2,b_3)$ and $(1,b_1)$ to obtain the contradiction
    $\nu(\cH([3])) \geq 4$.

    In the second case look at the $4$ sets $\{b_2,a_3,b_3\},\,
    (a_1,c_2),\, (a_2,c_3)$ and $(1,b_1)$ to get the same
    contradiction. (Let us remark that $(1,b_1) \in \cG$ follows from
    $g_1 = 9$.)

    Basically the same argument shows that none of the remaining subsets
    of $F_3 \cup \{b_2,c_2\} \cup \{b_1,c_1,d\}$ are in $\cF_1$.  This
    provides us with $\binom{8}{3}-1=55$ sets missing from $\cF_1$.
    Using~\eqref{eq:111} gives
    \[
    |\cF| \leq (165 - 55) + 18 \cdot 2 = 146 < 165. \qedhere
    \]
  \end{itemize}
\end{proof}

What we showed is that either $g_2 \leq 8$ or $(1,b_1) \notin \cG_1$
holds.

Plugging $g_1 + g_2 \leq 17$ back into~\eqref{eq:111} and using the
indirect assumption $|\cF| \geq 166$ gives
\begin{equation}
  \label{eq:114}
  |\cF_1| \geq 166 - 2 \cdot 17 = 132
\end{equation}

\begin{claim}
  \label{cl:114}
  $(4,5) \notin \cG$.
\end{claim}

\begin{proof}
  Suppose the contrary and let $\cP \subset \binom{[3] \cup [6,11]}{3}$
  be the collection of missing $3$-subsets.  In analogy with
  Claim~\ref{cl:112}, $|\cP| \geq 28$.

  If $P \in \cP$ and $P \cap [3] = \{\ell\}$ for some $\ell \in [3]$,
  then stability implies that both $(P-\{i\}) \cup \{4\}$ and
  $(P-\{i\}) \cup \{5\}$ are missing from $\cF$.

  Let us note that for $P \in \cP$, $P \cap [3] = (u,v)$ implies $(u,v)
  \notin \cG$.  However, even $(2,3) \notin \cG$ would imply $g_1 + g_2
  \leq 10$.  Consequently, $(P \cap [3]) \leq 1$ for all $P \in \cP$.

  On the other hand there can be at most $\binom{6}{2} = 15$ sets in
  $\cP$ that do not intersect $[3]$.  The remaining, at least 13, sets
  are of the form $(i,p,q)$ with $1 \leq i \leq 3, \ 6 \leq p < q \leq
  11$.  There are only 3 choices for $i$.  Thus there are at least $4$
  choices for $(p,q)$ such that $(i,p,q) \in \cP$ for at least one
  choice of $i \in [3]$.  For each of them stability implies $(4,p,q)
  \notin \cF$ and $(5,p,q) \notin \cF$.  Therefore at least $4 \cdot 2 =
  8$ new sets are excluded from $\binom{[11]}{3}$, making the total of
  $28 + 8 = 36$, and this contradicts~\eqref{eq:114}.
\end{proof}

Note that Claim~\ref{cl:114} shows that $G \cap [3] \not= \emptyset$ for
all $G \in \cG$.  This brings $\cF$ pretty close to $\cA_1(13)$.  Next
we show that, except for $(3,4)$ and $(3,5)$, there are no sets starting
with $3$.

\begin{claim}
  \label{cl:115}
  $(3,6) \notin \cG$.
\end{claim}

\begin{proof}
  Assume $(3,6) \in \cG$.  Consider now the family $\cP$ of missing
  $3$-sets in $\binom{[11]-(3,6)}{3}$.  Just as in Claim~\ref{cl:112},
  $|\cP| \geq 28$ holds.  Since $(1,5) \ll (2,5) \ll (3,6)$, both
  $(1,5)$ and $(2,5)$ are in $\cG$.  Thus there is no $P \in \cP$ with
  $|P \cap ((1,2) \cup (4,5))| \geq 2$, except possibly if
  $P \cap ((1,2) \cup (4,5)) = (4,5)$.  There can be at most
  $\bigl|[7,11]\bigr| = 5$ sets of the latter type.  There can be
  $\Bigl|\binom{[7,11]}{3}\Bigr| = 10$ sets in $\cP$ that do not
  intersect $(1,2) \cup (4,5)$.  For the remaining at least $28-15=13$
  sets $P \in \cP$ one has $\bigl|P \cap ((1,2) \cup (4,5))\bigr| = 1$.

  For a set of the form $(i,p,q) \in \cP$ with $i \in (1,2)$, $(p,q)
  \subset [7,11]$, note that $(3,p,q) \notin \cF_1$ holds by stability.
  Similarly if $i \in (4,5)$ then $(6,p,q) \notin \cF_1$ follows.  This
  way we associate the same missing new set with at most $2$ sets in
  $\cP$.  Thus we obtain at least $\lceil \frac{13}{2} \rceil = 7$ extra
  missing sets.  This brings the total to at least $28 + 7 = 35$, i.e.,
  $|\cF_1| \leq 165 - 35 = 130$, contradicting~\eqref{eq:114}.
\end{proof}

Inequality~\eqref{eq:114} shows that at most $165-132 = 33$ sets are
missing from $\binom{[11]}{3}$.  On the other hand, in
Claim~\ref{cl:112} we showed that at least 28 sets are missing from
$\binom{[3,11]}{3}$. This implies

\begin{claim}
  \label{cl:116}
  There are at most five $3$-element sets containing $1$ or $2$ that are
  missing from $\cF_1$.
\end{claim}

\begin{corollary}
  \label{cor:112}
  $(2,8,9) \in \cF_1$ and $(2,8,10) \in \cF_1$ unless \underline{all}
  $3$-sets containing $1$ are in $\cF_1$.
\end{corollary}

\begin{proof}
  There are $\binom{4}{2} = 6$ sets of the form $(2,a,b) : (a,b) \subset
  (8,9,10,11)$.  Using stability the statement follows.
\end{proof}

\begin{claim}
  \label{cl:117}
  $(5,6,7),\, (5,6,8) \in \cF_1$.
\end{claim}

\begin{proof}
  Since $\Bigl|\binom{[5,11]}{3}\Bigr| = 35$, at least 2 of these sets
  have to be in $\cF_1$.  The statement follows by stability.
\end{proof}

\begin{claim}
  \label{cl:118}
  $(3,4) \in \cG$.
\end{claim}

\begin{proof}
  Suppose the contrary.  Since $2 = a_1$ in our notation, we infer that
  all edges in $\cG$ contain either $1$ or $a_1$.  In particular,
  $\cG(2,3) = \emptyset$.  For $\cG_2(1,2)$ and $\cG_2(1,3)$ also, there
  can be a maximum of $3$ edges, namely the ones containing $a_1$.  Thus
  $g_2 \leq 6$.  Using Corollary~\ref{cor:111}, $g_2=6,\ g_1=9$ follow.
  In particular, $(a_1,c_2)$ and $ (a_2,c_3)$ are in $\cG$.
  Consequently, $(a_1,x) \notin \cG$ might be possible only for $x =
  b_1,c_1$ and $d$.

  Moreover, using $g_1=9$, either $(a_1,b_1)$ or $(1,c_1)$ is in $\cG$.
  Consequently, the $15$ edges in $\cG$ can be listed:
  \[
  \{(1,x) : 2 \leq x \leq 9 \} \cup \{(2,y) : 3 \leq y \leq 8\}
  \]
  along with either $(1,10)$ or $(2,9)$.  Plugging $g_1 + g_2 = 15$ once
  again into~\eqref{eq:111} gives:
  \[
  |\cF_1| \geq 166 - 2 \cdot 15 = 136 = \binom{11}{3} - 29.
  \]
  That is, except for the, at least 28, elements of $\binom{[3,11]}{3}$
  there is at most $1$ missing $3$-set from $\cF_1$.  By stability, only
  $(2,10,11)$ could be missing.  Thus $(1,10,11)$ and $(2,9,11)$ are in
  $\cF_1$.  Translating it to our special notation, $\{1,c_1,d\} \in \cF$
  and $\{2,b_1,d\} \in \cF$ follow.

  Now we can get easily $4$ pairwise disjoint sets:
  \[
  \begin{split}
    &F_2,\, F_3,\, \{1,c_1,d\},\, (a_1,b_1) \quad \text{or}\\
    &F_2,\, F_3,\, (a_1,b_1,d),\, (1,c_1), \quad \textrm{a contradiction}.
  \end{split}
  \]
\end{proof}

\begin{claim}
  \label{cl:119}
  $(1,7) \in \cG$.
\end{claim}

\begin{proof}
  Otherwise $\cG \subset \binom{[6]}{2}$.  Using $(4,5) \notin \cG$, $|\cG|
  \leq 14$ follows, a contradiction.
\end{proof}

\begin{claim}
  \label{cl:1110}
  $(2,x,y) \notin \cF$ for $(x,y) \subset (9,10,11)$.
\end{claim}

\begin{proof}
  $(1,7),\, (3,4),\, (5,6,8)$ and $(2,x,y)$ are $4$ pairwise disjoint sets.
\end{proof}

\begin{corollary}
  \label{cor:113}
  $(1,9,10) \in \cF_1$.
\end{corollary}

\begin{proof}
  Otherwise, by stability, all $3$ sets $(1,x,y)$ are missing from $\cF_1$,
  $(x,y) \subset (9,10,11)$.  Together we find six, that is more than five,
  missing $3$-sets containing $1$ or $2$, a contradiction
\end{proof}

\begin{corollary}
  \label{cor:114}
  $(2,7) \notin \cG$.
\end{corollary}

\begin{proof}
  The $4$ sets $(1,9,10),\, (2,7),\, (3,4)$ and $(5,6,8)$ are pairwise
  disjoint. 
\end{proof}

\begin{claim}
  \label{cl:1111}
  $\displaystyle \cG = \bigl\{ (3,x) : x=4,5 \bigr\} \cup \bigl\{(2,y) : 3
  \leq y \leq 6\} \cup \{(1,z) : 2 \leq z \leq 10 \bigr\}$.
\end{claim}

\begin{proof}
  The above $\cG$ has $15$ elements.  Now the statement follows from $|\cG|
  \geq 15$ and $(4,5) \notin \cG,\ (3,6) \notin \cG,\ (2,7) \notin \cG,\
  (1,d) \notin \cG$.
\end{proof}

\begin{claim}
  \label{cl:1112}
  $(5,7,8) \notin \cF_1$.
\end{claim}

\begin{proof}
  The $4$ sets $(1,9),\, (2,6),\, (3,4)$ and $(5,7,8)$ are pairwise
  disjoint.
\end{proof}

\begin{corollary}
  \label{cor:115}
  The following $30$ sets are missing from $\cF_1$:
  \[
  \binom{[7,11]}{3},\ \{(i,x,y) : i = 5,6;\ (x,y) \subset [7,11]\}.
  \]
\end{corollary}

\begin{proof}  
  By $(5,7,8) \notin \cF_1$ and stability.
\end{proof}

Finally, we can get the contradiction.  Corollary~\ref{cor:115} and
Claim~\ref{cl:1110} provide us with $33$ missing sets.  Now~\eqref{eq:111}
and $|\cG| = 15$ imply
\[
|\cF| \leq (165-33) + 2 \cdot 15 = 162 < 166
\]

\section{Uniqueness and beyond}\label{sec:beyond}

We did not explicitly state it, but the case of stable families, the
proof yields that $|\cF| = \max\bigl\{ \bigl| \cA_3 \bigr|,\, \bigl|
\cA_1(n) \bigr| \bigr\}$ is only possible if $\cF = \cA_3$ or $\cF =
\cA_1(n)$ holds.  Then it is not hard to show that even without assuming
stability, the families of maximal size are unique up to isomorphism.
For stable families our proof yields much more.

\begin{theorem}
  Let $\cF \subset \binom{[n]}{3}$ be a stable family with $\nu(\cF) =
  \nu(\cF(\bar{1})) = s,\ s \geq 5$.  Then
  \begin{equation}
    \label{eq:131}
    |\cF| \leq \max \Bigl\{ \bigl| \cF_3 \bigr|,\, \bigl| \cF_2(n)
    \bigr| \Bigr\} 
  \end{equation}
  holds and in case of equality $\cF = \cF_3$ or $\cF = \cF_2(n)$.
\end{theorem}

For the cases $s=2,3$ and $4$ the same result holds, but one has to do
an even more detailed case analysis (or find a different proof).

In this paper we prove some results for general $k$ but did not even
come close to giving a full proof of the Matching Conjecture.  Let us
announce two results which will appear in a forthcoming paper.

\begin{theorem}
  For $k=4$ and $s > s_0$ the Matching Conjecture is true.
\end{theorem}

For the second we need a definition.

Let $(x_0,x_1,\dots,x_{s-1}) \subset [n]$ and let $F_1,\dots,F_s$ be
pairwise disjoint sets, $x_i \in F_i,\ 1 \leq i < s$ but $x_0 \notin F_1
\cup \cdots \cup F_s \subset [n]$.  Define a graph $\cG$ with edge set
consisting of all $\{x_i,y_i\}$ satisfying $y_i \in F_{i+1} \cup \cdots
F_s,\ 0 \leq i < s$.  Finally define the $k$-graph $\cF(\cG)$ by
\[
\cF(\cG) = \biggl\{ F \in \binom{[n]}{k} : E \subset F \textrm{ holds
  for some edge } E \in \cG \biggr\} \cup \{F_1,F_2,\dots,F_s\}.
\]

\begin{theorem}
  Let $k \geq 4,\, n \geq n_0(k,s)$ and let $\cF \subset \binom{[n]}{k}$
  be a stable family with $\nu(\cF) = \nu(\cF(\bar{1})) = s$.  Then
  $\bigl| \cF \bigr| \leq \bigl| \cF(\cG) \bigr|$ and in case of
  equality $\cF$ is isomorphic to $\cF(\cG)$.
\end{theorem}

\begin{ack}
  The author is indebted to Tomasz \L{}uczak for sending him their
  paper~\cite{LM}, to Vojt\v{e}ch R\"odl for discussions and
  encouragement and to Steve La Fleur for his enormous help in preparing
  the TeX version of the paper.
\end{ack}

\printbibliography
\end{document}